\documentclass[reqno,a4paper, 11pt]{amsart}

\numberwithin{equation}{section}

\usepackage{amsfonts,amssymb,amsmath,amsthm}
\usepackage{enumerate}
\usepackage{bbm}
\usepackage{times}
\usepackage{amsfonts}
\usepackage{amssymb}
\usepackage{bbm}
\usepackage{times}
\usepackage{amsfonts}
\usepackage{amssymb}
\usepackage{bbm}
\usepackage{times}

\usepackage{enumerate}
\usepackage{amsmath}
\usepackage{amsthm}
\usepackage{color}
\usepackage [latin1]{inputenc}
\newtheorem{theor}{Theorem}[section]

\newtheorem{limma}[theor]{Lemma}

\newtheorem{prop}[theor]{Proposition}

\newcounter{other}            
\newtheorem{otherth}[other]{Theorem}              
\newtheorem{otherp}[other]{ Proposition}
\newtheorem{otherl}[other]{ Lemma}        

\def\B{\mathcal{B}}
\def\D{\mathbb{D}}
\def\C{\mathbb{C}}
\def \T {\mathbb{T}}
\def \f {\frac}
\def \ind {\int_\D}

\def \d {\mathcal D}
\def \ind {\int_\D}
\def \o {\omega}
\def \op {{\omega_{[p-2]}}}
\def \vp {\varphi}
\def \ol {\overline}
\def \p {\prime}
\def \hD {\hat{\mathcal{D}}}
\begin{document}

\title[Double integral estimates for Besov type spaces and their  applications]
{Double integral estimates for Besov type spaces and their  applications}

\author{Guanlong Bao}
\address{Department of Mathematics\\
         Shantou University\\
         Shantou 515063, Guangdong, China}
\email{glbao@stu.edu.cn}

\author{Juntao Du}
\address{Department of Mathematics\\
         Shantou University\\
         Shantou 515063, Guangdong, China}
\email{jtdu007@163.com}

\author{Hasi Wulan}
\address{Department of Mathematics\\
         Shantou University\\
         Shantou 515063, Guangdong, China}
\email{wulan@stu.edu.cn}

\thanks{The work was supported by NNSF of China (No. 11720101003) and Guangdong basic and applied basic research foundation (No. 2022A1515012117). }

\begin{abstract} For  $0<p<\infty$, we give a complete description of nonnegative radial weight functions $\omega$ on the open unit disk $\mathbb{D}$ such that
$$
 \int_{\mathbb{D}} |f'(z)|^p (1-|z|^2)^{p-2}\omega(z)dA(z)<\infty
 $$
if and only if
$$
\int_{\mathbb{D}}\int_{\mathbb{D}}\frac{|f(z)-f(\zeta)|^p}{|1-\overline{\zeta}z|^{4+\tau+\sigma}}(1-|z|^2)^{\tau}(1-|\zeta|^2)^{\sigma}\omega(\zeta)dA(z)A(\zeta)<\infty 
$$
for all analytic functions $f$ in $\mathbb{D}$, where  $\tau$ and $\sigma$ are some real numbers.  As applications, we give some geometric descriptions of functions in Besove type spaces $B_p(\omega)$ with doubling weights,  and characterize the boundedness and compactness of Hankel type operators related to Besov type spaces with radial B\'ekoll\'e-Bonami weights. Some special cases of our results  are new even for some standard weighted Besov  spaces.

\end{abstract}

\maketitle

\section{Introduction}

A classical topic in complex analysis and related field  is to study double integral estimates for Dirichlet type spaces or  Besove type spaces.
 Recall that  the Dirichlet integral of a function $f\in L^2(\T)$ is
 $$
\mathfrak{D}(f)=\int_\T \int_\T \left|\frac{f(\zeta)-f(\eta)}{\zeta-\eta}\right|^2 |d \zeta| |d\eta|,
$$
 where $\T$ is the boundary of the open unit disk $\D$ in the complex plane $\C$. By this integral, in 1931 J. Douglas \cite{Do} studied  the theory of   minimal surfaces, and in 1940 A. Beurling \cite{Be} proved that the Fourier series of a  Dirichlet function converges everywhere  except on a set of logarithmic capacity zero. Let $H(\D)$ be the space of analytic functions in   $\D$. The Dirichlet space  $\mathfrak{D}$  consists of functions $f\in H(\D)$ such that
$$
\ind |f'(z)|^2 dA(z)<\infty,
$$
where $dA(z)=(1/\pi) dxdy$ is the normalized area measure on $\D$.  It is known that
every function $f$ in $\mathfrak{D}$ has nontangential limit denoted by $f(\zeta)$ for almost every $\zeta\in \T$.
For   $f\in H(\D)$, it is  also well known that $f\in \mathfrak{D}$ if and only if $\mathfrak{D}(f)<\infty$.

For $p>1$, the Besov space $B_p$ is the space of functions $f\in H(\D)$ satisfying
$$
\ind |f'(z)|^p (1-|z|^2)^{p} d\lambda(z)<\infty,
$$
where
$$
d\lambda(z)=\frac{dA(z)}{(1-|z|^2)^2}
$$
is the M\"obius invariant measure on $\D$.  Clearly, $B_2$ is equal to the Dirichlet space $\mathfrak{D}$.
For $\alpha>-1$ and $p>\max\{1, 2/(2+\alpha)\}$, in 1988 J. Arazy, S. Fisher and J. Peetre (cf. \cite[Theorem 6.4]{AFP} and \cite[Remark, p. 1026]{AFP}) showed  that a  function $f\in B_p$ if and only if  the following  integral
$$
\ind \left(\ind |f\circ \vp_a(z)-f(a)|^2 (1-|z|^2)^\alpha dA(z)\right)^{p/2} d\lambda(a)
$$
 is finite    for $f\in H(\D)$,
where  $\vp_a(z)=(a-z)/(1-\overline{a}z)$ is a  M\"obius map interchanging the points $0$ and $a$. In \cite{AFP}, this characterization  was used to investigate
the properties of Hankel operators on weighted Bergman spaces. In 1991 K. Zhu \cite{Zhu1} gave another double integral estimates for Besov spaces; that is, if  $f\in H(\D)$, then
$f\in B_p$ if and only if
$$
\ind \ind \frac{|f(z)-f(w)|^p }{|1-\overline{z}w|^4} dA(z)dA(w)<\infty.
$$

Recall that for $p>0$ and $\alpha<p/2$, the standard weighted   Besov  space $B_p(\alpha)$ consists of functions $f\in H(\D)$ for which $$
\ind |f'(z)|^p (1-|z|^2)^{p-1-2\alpha} dA(z)<\infty.
$$
If $p=2$, then $B_p(\alpha)$ is the Dirichlet type space $\mathfrak{D}_\alpha$.
In 1993 R. Rochberg and Z. Wu \cite{RW} obtained   a double integral characterization of Dirichlet type spaces $\mathfrak{D}_\alpha$ and applied this characterization to study   Hankel type operators. In 2008 D. Blasi and J. Pau
\cite{BP} generalized this characterization   from $\mathfrak{D}_\alpha$ to  $B_p(\alpha)$ by a different method.  It is known from \cite[Theorem 2.2]{BP} that if  $f\in H(\D)$, $p>1$, $\sigma, \tau>-1$ and
$\alpha\leq 1/2$ such that $\min(\sigma, \tau)+2\alpha>-1$, then $f\in B_p(\alpha)$ if and only if
$$
\ind\ind \frac{|f(z)-f(w)|^p}{|1-\overline{w}z|^{3+\sigma+\tau+2\alpha}}(1-|w|^2)^\sigma (1-|z|^2)^\tau dA(z)dA(w)<\infty.
$$
This double integral characterization was also used  to study  Hankel type operators; see \cite[Section 4]{BP}. We refer to \cite[Section 4]{BGP} and \cite{YZ} for some recent results associated with  double integral estimates for  some analytic function spaces. See A.  Reijonen \cite{Re} for recent  results on Besov type spaces induced by some radial weights.

A function $\o: \D\rightarrow [0, \infty)$, integrable over $\D$, is called a weight. If $\o(z)=\o(|z|)$ for all $z\in \D$, then we say that $\o$ is radial. Suppose  $0<p<\infty$ and $(1-|z|^2)^{p-2}\o (z)$ is a weight on $\D$. Denote by $B_p(\o)$ the Besov type space consisting of those functions $f\in H(\D)$ such that
$$
\|f\|_{B_p(\o)}=|f(0)|+ \left(\int_{\mathbb{D}} |f'(z)|^p (1-|z|^2)^{p-2}\omega(z)dA(z)\right)^{1/p}<\infty.
 $$
 If  $\omega(z)=(1-|z|^2)^{1-2\alpha}$, then $B_p(\o)=B_p(\alpha)$.

In this paper, for all $0<p<\infty$,  we give a complete description of  radial weights  $\omega$ such that an analytic function $f$ belongs to $B_p(\o)$ if and only if $f$ satisfies certain double integral estimates. This result is  new even  for  $B_p(\alpha)$ when  $0<p\leq 1$ and $\alpha<p/2$.  As applications, we consider geometric descriptions  of functions in Besove type space with doubling weights.  We also apply these double integral estimates  to characterize  the boundedness of  Hankel type operators from  $B_p(\o)$  with  radial B\'ekoll\'e-Bonami weights  to the corresponding nonanalytic version of   Besov type spaces, which  is also new for $B_p(\alpha)$ when  $1<p<\infty$ and $1/2<\alpha <p/2$. The compactness of these Hankel type operators is  also investigated.

Throughout  this paper,   we write $a \lesssim b$ if   there exists a positive constant $C$ such that $a \leq Cb$. If  $a\lesssim b\lesssim a$, then we write  $a \thickapprox b$.

\section{Double integral estimates for $B_p(\o)$ spaces }

This section is devoted to investigate double integral estimates for $B_p(\o)$ spaces. In particular, for  $0<p<\infty$  we characterize completely radial weights $\omega$ such that
\begin{align*}
&\int_{\mathbb{D}} |f'(z)|^p (1-|z|^2)^{p-2}\omega(z)dA(z)\\
\thickapprox &\int_{\mathbb{D}}\int_{\mathbb{D}}\frac{|f(z)-f(\zeta)|^p}{|1-\overline{\zeta}z|^{4+\tau+\sigma}}(1-|z|^2)^{\tau}(1-|\zeta|^2)^{\sigma}\o(\zeta)dA(z)A(\zeta)
\end{align*}
for all $f\in H(\D)$, where $\sigma$ and $\tau$ are  real numbers  in  certain  ranges.  This  result is new even for the space  $B_p(\alpha)$ when  $0<p\leq 1$ and $\alpha<p/2$.
For  $p=2$, the  double integral estimates for $B_p(\o)$ spaces   also complete corresponding conclusions   in \cite[p. 1725]{BLQW} and  \cite[p. 210]{LQ},  respectively.

\subsection{Littlewood-Paley estimates for weighted Bergman spaces and a Forelli-Rudin type  estimate} In this subsection, we recall some Littlewood-Paley estimates for weighted Bergman spaces. We also give an elementary proof of a well-known Forelli-Rudin type  estimate. All of these estimates are  tools to prove our main theorems   in this section.

Suppose $0<p<\infty$ and $\mu$ is a nonnegative Borel measure on $\D$. The  Lebesgue space $L^p(\D, d\mu)$ consists of   complex-valued measurable functions  $f$ on $\D$ for which
$$
\|f\|_{L^p(\D, d\mu)}=\left(\ind |f(z)|^p d\mu(z)\right)^{\frac{1}{p}}<\infty.
$$
For $0<p<\infty$ and a weight $\o$ in $\D$, the weighed Bergman space $A^p_\o$ is the space of analytic functions in $L^p(\D, \o dA)$. If $\o(z)=(\alpha+1)(1-|z|^2)^\alpha$, $\alpha>-1$, then $A^p_\o$ is the standard weighted Bergman space $A^p_\alpha$.

For $f\in H(\D)$,  it is well known that the growth rate of the following functions are often comparable in some sense:
$$
f(z), (1-|z|^2)f'(z), (1-|z|^2)^2f''(z), (1-|z|^2)^3f'''(z), \cdots.
$$
This kind of estimates is usually called Littlewood-Paley estimates.
We refer to \cite{AC, APR, AS, BWZ, PP, PR2021} for the study of Littlewood-Paley estimates for $A^p_\o$; that is,  to characterize weight functions $\o$ such that
$$
\ind |f(z)-f(0)|^p \o(z)dA(z)\thickapprox \ind |f'(z)|^p (1-|z|^2)^p  \o(z)dA(z)
$$
for all $f\in H(\D)$.  When weight  $\o$ is  radial,  this question has been recently solved completely  in \cite{PR2021}.

The pseudo-hyperbolic metric on $\D$ is defined by
$$
\rho(z, w)=\left|\frac{z-w}{1-\overline{z}w}\right|, \ \ z, \ w\in \D.
$$
For $0<r<1$ and $a\in \D$, denote by
$$
\Delta(a, r)=\{z\in \D:  \rho(a, z)<r\}
$$
the  pseudo-hyperbolic disk of center $a$ and radius $r$.

The following Littlewood-Paley estimates for $A^p_\o$  \cite{BWZ} will be    useful in this  paper. For a given  weight $\o$,  the conditions appeared   in the following theorems are easy to verify.

\begin{otherth}\label{1Bao etac}
Suppose $p>0$ and $\omega$ is a weight. If there
exist two constants $r\in(0,1)$ and $C>0$ such that
$$
C^{-1}\omega(\zeta)\le\omega(z)\le C\omega(\zeta)
$$
for all $z$ and $\zeta$ satisfying $\rho(z,\zeta)<r$, then there exists another positive
constant $C$ such that
$$
\ind(1-|z|^2)^p|f'(z)|^p\omega(z)\,dA(z)\le C\ind|f(z)-f(0)|^p\omega(z)\,dA(z)
$$
for all $f\in H(\D)$.
\end{otherth}

\begin{otherth}\label{2Bao etac}
Suppose $p>0$, $\omega$ is a weight, and there exist $t_0\ge0$
and $s_0\in[-1,0)$ with the following property: for any $t>t_0$ and $s>s_0$ there
is a positive constant $C=C(t,s)$ such that
$$
\ind\frac{\omega(\zeta)(1-|\zeta|^2)^s\,dA(\zeta)}{|1-z\overline \zeta|^{2+s+t}}\le
\frac{C\omega(z)}{(1-|z|^2)^t}
$$
for all $z\in\D$. Then there exists another positive constant $C$ such that
$$
\ind|f(z)-f(0)|^p\omega(z)\,dA(z)\le C\ind(1-|z|^2)^p|f'(z)|^p\omega(z)\,dA(z)
$$
for all $f\in H(\D)$.
\end{otherth}

We recall some  well-known estimates  as follows (see \cite[Lemma 3.10]{Zhu}).
\begin{otherl}\label{well-known estimate}
 Let $\beta$ be any real number.  Then
$$
\int^{2\pi}_0\frac{d\theta}{|1-ze^{-i\theta}|^{1+\beta}}\thickapprox
\begin{cases}1 & \enspace \text{if} \ \ \beta<0,\\
                     \log\frac{2}{1-|z|^2} & \enspace  \text{if} \ \  \beta=0,\\
                     \frac{1}{(1-|z|^2)^\beta} & \enspace \text{if}\ \  \beta>0,
                   \end{cases}
$$
for all $z\in \D$. Also,  suppose  $c$ is real and  $t>-1$. Then
$$
\int_\D\frac{(1-|w|^2)^t}{|1-\bar{z}w|^{2+t+c}}dA(w)\thickapprox
\begin{cases}1 & \enspace \text{if} \ \ \ c<0,\\
                     \log\frac{2}{1-|z|^2} & \enspace  \text{if} \ \  \ c=0,\\
                     \frac{1}{(1-|z|^2)^c} & \enspace \text{if}\ \  \ c>0,
                   \end{cases}
$$
for all $z\in \D$.
\end{otherl}

The following Forelli-Rudin type  estimate  \cite{OF} is  very   useful in the analysis of some  function spaces. A rather complicated proof of this  estimate  was given  in \cite{Zhao}. Later,  for $s>-1$, $r>0$, $t>0$, and  $t+1<s+2<r$, a simple proof was given in \cite[pp. 27-28]{WZ}, where the case of $t<s+2\leq t+1<r$ is missing.
\begin{otherl} \label{Forelli-Rudin type  estimates}
Suppose $s>-1$, $r>0$, $t>0$, and  $t<s+2<r$. Then there exists a positive constant $C$ such that
$$
\ind\frac{(1-|w|^2)^s}{|1-\overline{w}z|^r |1-\overline{w}\zeta|^t} dA(w) \leq
C\,\frac{(1-|z|^2)^{2+s-r}}{|1-\overline{\zeta}z|^t}
$$
for all $z,\zeta\in\D$.
\end{otherl}
We present an elementary proof of Lemma \ref{Forelli-Rudin type  estimates} here.
For  $z$, $\zeta\in\D$, by a change of variables $w=\varphi_z(a)$ and Lemma \ref{well-known estimate}, we see that
\begin{align*}
&\ind\frac{(1-|w|^2)^s}{|1-\overline{w}z|^r |1-\overline{w}\zeta|^t} dA(w)\\
=&(1-|z|^2)^{s+2-r}\ind \frac{ (1-|a|^2)^sdA(a)}{|1-\overline{z}a|^{4+2s-r}|1-\overline{\zeta }\varphi_z(a)|^t}\\
=& \frac{(1-|z|^2)^{s+2-r}}{|1-\overline{z}\zeta|^t} \ind \frac{(1-|a|^2)^sdA(a)}{|1-\overline{z}a|^{4+2s-r-t}|1-a \overline{\varphi_z(\zeta)}|^t}\\
\leq & \frac{(1-|z|^2)^{s+2-r}}{|1-\overline{z}\zeta|^t} \int_{\{a\in \D: |1-\overline{z}a|\geq  |1-a \overline{\varphi_z(\zeta)}|\}} \frac{(1-|a|^2)^sdA(a)}{|1-\overline{z}a|^{4+2s-r-t}|1-a \overline{\varphi_z(\zeta)}|^t}\\
&+  \frac{(1-|z|^2)^{s+2-r}}{|1-\overline{z}\zeta|^t} \int_{\{a\in \D: |1-\overline{z}a|\leq   |1-a \overline{\varphi_z(\zeta)}|\}} \frac{(1-|a|^2)^sdA(a)}{|1-\overline{z}a|^{4+2s-r-t}|1-a \overline{\varphi_z(\zeta)}|^t}\\
\lesssim &  \frac{(1-|z|^2)^{s+2-r}}{|1-\overline{z}\zeta|^t} \left(\sup_{b\in \D}\ind \frac{(1-|a|^2)^sdA(a)}{|1-\overline{a}b|^{4+2s-r}}+ \sup_{b\in \D}\ind \frac{(1-|a|^2)^sdA(a)}{|1-\overline{a}b|^{t}}\right)\\
\lesssim &  \frac{(1-|z|^2)^{s+2-r}}{|1-\overline{z}\zeta|^t},
\end{align*}
 which finishes the proof of Lemma \ref{Forelli-Rudin type  estimates}.

\subsection{Weighted Bergman spaces and doubling weights}

In this subsection, we recall characterizations of Carleson measures for $A^p_\o$ induced by doubling weights and also give some estimates related to these weights.

 Let $\hat{\d}$ be the class of radial weights $\o$ on $\D$ for which $\hat{\o}(r)=\int_r^1 \o(s)ds$ admits the doubling property $\hat{\o}(r) \leq C \hat{\o}(\frac{1+r}{2})$ for all $r\in [0, 1)$, where $C=C(\o)>1$. A weight in $\hat{\d}$ is usually  called a doubling weight. If there exist
 $K=K(\o)>1$ and $C=C(\o)>1$  such that the radial weight $\o$ satisfies
 $$
 \hat{\o}(r)\geq C \hat{\o}\left(1-\frac{1-r}{K}\right), \ \ 0\leq r<1,
 $$
then we say that $\o\in \check{\d}$. The intersection $\hat{\d} \cap \check{\d}$ is denoted by $\d$.
The classes of these weights arose naturally in the study of some analytic function spaces and related operator theory. For instance, by \cite{PR2021}, for a radial weight $\o$, certain  Bergman projection $P_\o$ is bounded from $L^\infty$  to the Bloch space $\B$ if and only if $\o \in \hat{\d}$; $P_\o: \ L^\infty\rightarrow \B$ is bonded and onto if and only if $\o \in \d$; the classes $\hat{\d}$ and $\d$ also characterize completely Littlewood-Paley estimates for weighted Bergman spaces with radial weights.   See \cite{Pe, PR2014}  for properties of these weights. Our investigation in  this section will guide us to  find new and more  significance of the class   $\hat{\d}$; that is, the weight in  $\hat{\d}$ describes precisely certain double integral estimate for  Besov type  spaces.

For a   space $X$ of analytic functions on $\D$ and $0<p<\infty$, a nonnegative Borel measures $\mu$ on $\D$ is said to be a $p$-Carleson measure for $X$ if the identity operator $I_d:X\to L^p(\D, d\mu)$ is bounded; that is,
$$
\left(\ind |f(z)|^p d\mu(z)\right)^{\frac{1}{p}} \lesssim \|f\|_{X}
$$
for all $f\in X$.  If $I_d:X\to L^p(\D, d\mu)$ is compact, then we say that $\mu$ is a vanishing $p$-Carleson measure for $X$.

 For a  doubling weight $\o$ and   $0<p, q<\infty$,  J. Pel\'aez and J. R\"atty\"a \cite{PR} characterized nonnegative Borel measures $\mu$ on $\D$ such that the differentiation operator of order $n\in \mathbb{N} \cup \{0\}$ is bounded  from $A_{\o}^p$ into $L^q(\D, d\mu)$. In particular, they gave  the following result.

\begin{otherth} \label{CM for Apw}
Suppose $0<p<\infty$, $\o\in \hat{\d}$ and $\mu$ is a nonnegative Borel measure on $\D$. Then $\mu$ is a $p$-Carleson measure for $A^p_\o$ if and only if the function
$$
a\mapsto \frac{\mu(S(a))}{\int_{S(a)} wdA}, \ \ a\in \D\setminus \{0\},
$$
is essentially bounded,  where
$$
S(a)=\left\{z\in \D: \left|\frac{\arg z-\arg a}{2\pi}\right|<\frac{1-|a|}{2}, \ |z|\geq |a|\right\}
 $$
is the Carleson box with vertex at $a$.
\end{otherth}

The following characterizations of doubling weights can be found in  \cite{Pe}.

\begin{otherl} \label{property of doub w}
Suppose $\o$ is a radial weight. Then the following conditions are equivalent:
\begin{enumerate}[(i)]
  \item  $\o \in \hat \d$;
  \item  there exists  a positive constant  $\beta$ depending only on  $\o$ such that
  \begin{align}\label{1023-1}
  \frac{\hat{\o}(r)}{(1-r)^\beta}\lesssim \frac{\hat{\o}(t)}{(1-t)^\beta}
  \end{align}
  for all  $0\leq r\leq t<1$;
  \item  there exists a positive constant   $\gamma$  depending only on  $\o$  such that
  \begin{align}\label{1023-2}
  \int_0^t \frac{\o (s)}{(1-s)^\gamma}ds\lesssim \frac{\hat \o(t)}{(1-t)^\gamma}
  \end{align}
  for all   $0\leq t<1$;
  \item there exists a nonnegative constant  $\lambda$ depending only on  $\o$  such that
 \begin{align}\label{1023-3}
  \int_\D \frac{\o(z)dA(z)}{|1-\overline{\xi}z|^{\lambda+1}}\lesssim \frac{\hat{\o}(\xi)}{(1-|\xi|)^\lambda}
  \end{align}
  for all $\xi\in\D$;
  \item  there exists a positive constant $\eta$ depending only on  $\o$  such that
\begin{align}\label{1023-4}
  \int_0^1 s^x\o(s)ds\lesssim \left(\frac{y}{x}\right)^\eta \int_0^1 s^y \o(s)ds
  \end{align}
   for all  $0<x\leq y<\infty$.
\end{enumerate}
\end{otherl}

 Of course,   constants $\beta$, $\gamma$, $\lambda$ and $\eta$ in Lemma \ref{property of doub w} are  not unique. In fact, if   $\beta$  satisfies  (ii) in  Lemma \ref{property of doub w}, then any  constant  bigger than $\beta$ also satisfies    (ii).  The same phenomenon occurs for $\gamma$, $\lambda$ and $\eta$.
Now we give  the following  relation among these parameters. The infimum below is useful in  our results.

\begin{limma} \label{Inf U}
Suppose $\o \in \hat \d$. Then
\begin{align*}
&\inf\{\beta: \ \beta\  \text{satisfies  (ii) in Lemma \ref{property of doub w}}  \}\\
=&\inf \{\gamma: \ \gamma\  \text{satisfies (iii) in Lemma \ref{property of doub w}}  \}\\
=&\inf \{\lambda: \ \lambda \  \text{satisfies (iv) in Lemma \ref{property of doub w}}  \}\\
=&\inf \{\eta: \ \eta \  \text{satisfies (v) in Lemma \ref{property of doub w}}  \}.
\end{align*}
\end{limma}

\begin{proof}
Checking that proof of Lemma \ref{property of doub w}  \cite{Pe},  we get
$$
\{\gamma: \ \gamma\  \text{satisfies (iii) in Lemma \ref{property of doub w}}\}\subseteq  \{\lambda: \ \lambda \  \text{satisfies (iv) in Lemma \ref{property of doub w}}  \},
$$
which yields
\begin{align}\label{1702}
&\inf \{\lambda: \ \lambda \  \text{satisfies (iv) in Lemma \ref{property of doub w}}  \} \nonumber \\
\leq &\inf \{\gamma: \ \gamma\  \text{satisfies (iii) in Lemma \ref{property of doub w}}\}.
\end{align}

Now suppose $\beta$  satisfies  (ii) in Lemma \ref{property of doub w}. Then for any $\epsilon>0$,  $\beta+\epsilon$ also  satisfies  the same property. Hence
 $$
\hat{\o}(0) \lesssim \frac{\hat{\o}(t)}{(1-t)^{\beta+\epsilon}}
$$
for all  $0\leq t<1$. Consequently,
\begin{align}
\int_0^t \frac{\o(s) }{(1-s)^{\beta+\epsilon}}ds
=& \hat{\o}(0)-\frac{\hat{\o}(t)}{(1-t)^{\beta+\epsilon}} +(\beta+\epsilon)\int_0^t \frac{\hat{\o}(s)}{(1-s)^{\beta+\epsilon+1}}ds \nonumber  \\
\lesssim & \frac{\hat{\o}(t)}{(1-t)^{\beta+\epsilon}}+(\beta+\epsilon) \frac{\hat{\o}(t)}{(1-t)^{\beta}} \int_0^t \frac{1}{(1-s)^{\epsilon+1}}ds \nonumber  \\
\lesssim & \left(1+\frac{\beta+\epsilon}{\epsilon}\right) \frac{\hat{\o}(t)}{(1-t)^{\beta+\epsilon}}, \label{7100}
\end{align}
which means that  when $\beta$  satisfies  (ii) in Lemma \ref{property of doub w},
$$
\beta+\epsilon \in \{\gamma: \ \gamma\  \text{satisfies (iii) in Lemma \ref{property of doub w}}  \}
$$
for any $\epsilon>0$.  Thus
\begin{align}\label{1703}
&\inf \{\gamma: \ \gamma\  \text{satisfies (iii) in Lemma \ref{property of doub w}}  \} \nonumber \\
 \leq & \inf\{\beta: \ \beta\  \text{satisfies  (ii) in Lemma \ref{property of doub w}}  \}.
\end{align}

Next let $\lambda$ satisfy  (iv) in Lemma \ref{property of doub w}. Without loss of generality, we can assume $\lambda>0$. In fact, if $\lambda=0$, it is enough to consider the case of $\lambda+\epsilon$ for any $\epsilon>0$.
For $0\leq r\leq t<1$, take $\xi \in \D$ such that $|\xi|=t$. It follows from  (\ref{1023-3}) and Lemma \ref{well-known estimate}  that
\begin{align*}
 \frac{\hat{\o}(t)}{(1-t)^\lambda} \gtrsim \int_\D \frac{\o(z)dA(z)}{|1-\overline{\xi}z|^{\lambda+1}}\gtrsim \int_r^1 \frac{\o(s)}{(1-st)^\lambda}ds\gtrsim \frac{\hat{\o}(r)}{(1-r^2)^\lambda},
\end{align*}
which gives that $\lambda \in \{\beta: \ \beta\  \text{satisfies  (ii) in Lemma \ref{property of doub w}}  \}$. Hence
\begin{align}\label{1704}
&\inf\{\beta: \ \beta\  \text{satisfies  (ii) in Lemma \ref{property of doub w}}  \} \nonumber \\
 \leq & \inf \{\lambda: \ \lambda \  \text{satisfies (iv) in Lemma \ref{property of doub w}}  \}.
\end{align}

By    (\ref{1702}), (\ref{1703}) and (\ref{1704}), we get
\begin{align}
  &\inf\{\beta: \ \beta\  \text{satisfies  (ii) in Lemma \ref{property of doub w}}  \} \nonumber \\
=&\inf \{\gamma: \ \gamma\  \text{satisfies (iii) in Lemma \ref{property of doub w}}  \} \label{1705} \\
=&\inf \{\lambda: \ \lambda \  \text{satisfies (iv) in Lemma \ref{property of doub w}}  \} \nonumber.
\end{align}

Let $\beta$ satisfy  (ii) in Lemma \ref{property of doub w}, so is $\beta+\varepsilon$ for any $\varepsilon>0$.  Consider  $2<x\leq y<\infty$ first. For $0<s<1$ the function $f_{x, \beta+\varepsilon}(s)=x^{\beta+\varepsilon} s^x(1-s)^{\beta+\varepsilon}$ takes its maximum at $s=x/(x+\beta+\varepsilon)$.
Thus when $s\in(0,1)$,
$$
x^{\beta+\varepsilon} s^x(1-s)^{\beta+\varepsilon}\leq  \left(\frac{x}{x+\beta+\varepsilon}\right)^x\left(\frac{x(\beta+\varepsilon)}{x+\beta+\varepsilon}\right)^{\beta+\varepsilon} \leq  (\beta+\varepsilon)^{\beta+\varepsilon}.
$$
Combining this with (\ref{7100}) and  (\ref{1023-1}), we deduce
\begin{align*}
&x^{\beta+\varepsilon} \int_0^1 s^x\o(s)ds\\
=&x^{\beta+\varepsilon} \int_{1-\frac{1}{x}}^1 s^x\o(s)ds +x^{\beta+\varepsilon} \int_0^{1-\frac{1}{x}}s^x\o(s)ds\\
\lesssim & \frac{\hat{\o}(1-\frac{1}{x})}{(1-(1-\frac{1}{x}))^{\beta+\varepsilon}}  +\int_0^{1-\frac{1}{x}}x^{\beta+\varepsilon}s^x(1-s)^{\beta+\varepsilon} \frac{\o(s)}{(1-s)^{\beta+\varepsilon}}ds\\
\lesssim &  \frac{\hat{\o}(1-\frac{1}{x})}{(1-(1-\frac{1}{x}))^{\beta+\varepsilon}}  +\int_0^{1-\frac{1}{x}} \frac{\o(s)}{(1-s)^{\beta+\varepsilon}}ds\\
\lesssim  & \frac{\hat{\o}(1-\frac{1}{x})}{(1-(1-\frac{1}{x}))^{\beta+\varepsilon}}
\lesssim \frac{\hat{\o}(1-\frac{1}{y})}{(1-(1-\frac{1}{y}))^{\beta+\varepsilon}}.
\end{align*}
Also,
\begin{align*}
 \frac{\hat{\o}(1-\frac{1}{y})}{(1-(1-\frac{1}{y}))^{\beta+\varepsilon}}
\thickapprox y^{\beta+\varepsilon}\int_{1-\frac{1}{y}}^1 s^y\o(s)ds
\lesssim y^{\beta+\varepsilon} \int_0^1 s^y\o(s)ds.
\end{align*}
Thus for $2<x\leq y<\infty$, one gets
\begin{equation}\label{1706}
 \int_0^1 s^x\o(s)ds \lesssim \left(\frac{y}{x}\right)^{\beta+\varepsilon} \int_0^1 s^y\o(s)ds.
 \end{equation}
When $0<x \leq y\leq 2$, (\ref{1706}) also  holds because
$$
\int_0^1 s^x\o(s)ds \thickapprox \int_0^1 s^y \o(s)ds\thickapprox \int_0^1 \o(s)ds.
 $$
When $0<x \leq 2<y<\infty$,  there is a small enough positive constant $c$ with $2<2+c<y$. Then
\begin{align*}
\int_0^1 s^x\o(s)ds\thickapprox \int_0^1 s^{2+c}\o(s)ds&\lesssim \left(\frac{y}{2+c}\right)^{\beta+\varepsilon}\int_0^1 s^y\o(s)ds\\
&\lesssim \left(\frac{y}{x}\right)^{\beta+\varepsilon}\int_0^1 s^y\o(s)ds.
\end{align*}
Consequently, $\beta+\varepsilon \in \{\eta: \ \eta \  \text{satisfies (v) in Lemma \ref{property of doub w}}  \}$. Then
\begin{align}\label{1707}
&\inf \{\eta: \ \eta \  \text{satisfies (v) in Lemma \ref{property of doub w}}  \} \nonumber \\
 \leq & \inf\{\beta: \ \beta\  \text{satisfies  (ii) in Lemma \ref{property of doub w}}  \}.
\end{align}
Conversely, suppose (v) in Lemma \ref{property of doub w} holds for some positive constant $\eta$. Then $\o\in\hD$ and hence there exits a positive constant $C=C(\o)>1$ such that
$$\hat{\o}(r)\leq C \hat{\o}\left(\frac{r+1}{2}\right), $$
for all $r\in(0,1)$.  Let $t\in(\frac{3}{4},1)$ and  $t_n=1-2^n(1-t)$, $n=0,1,\cdots,N-1$. Here $N=N(t)$ is the largest positive integer  such that $t_{N-1}> 0$. We set $t_N=0$.
Then
\begin{align*}
\int_0^t s^\frac{1}{1-t}\o(s)ds
&\thickapprox\sum_{n=0}^{N-1}\int_{t_{n+1}}^{t_n} s^\frac{1}{1-t}\o(s)ds\\
&\lesssim \sum_{n=0}^{N-1}\left(t_n^\frac{1}{1-t}C^{n}\hat{\o}(t) \right)\\
&\lesssim\left(\sum_{n=0}^{N-1}2^{-2^n}C^{n}\right)\hat{\o}(t)
\lesssim \hat{\o}(t),
\end{align*}
where we use
\begin{align*}
t_n^\frac{1}{1-t}
&=\left(\big(1-2^n(1-t)\big)^\frac{1}{2^n(1-t)}\right)^{2^n}\lesssim  2^{-2^n}
\end{align*}
for all $n=0,1,\cdots,N$ and $t\in(\frac{3}{4},1)$.
Hence,
\begin{align*}
\int_0^1 s^\frac{1}{1-t}\o(s)ds
=\int_0^t  s^\frac{1}{1-t}\o(s)ds +\int_t^1 s^\frac{1}{1-t}\o(s)ds
\lesssim \hat{\o}(t).
\end{align*}
Hence, for  $\frac{4}{3}\leq x\leq y<\infty$, we get
\begin{align*}
\frac{\hat{\o}(1-\frac{1}{x})}{(1-(1-\frac{1}{x}))^\eta}
&\lesssim x^\eta\int_0^1 s^x\o(s)ds\lesssim y^\eta\int_0^1 s^y\o(s)ds\\
&=\frac{\int_0^1 s^\frac{1}{1-(1-\frac{1}{y})}\o(s)ds}{(1-(1-\frac{1}{y}))^\eta}
\lesssim \frac{\hat{\o}(1-\frac{1}{y})}{(1-(1-\frac{1}{y}))^\eta},
\end{align*}
which gives
 $$
  \frac{\hat{\o}(r)}{(1-r)^\eta}\lesssim \frac{\hat{\o}(t)}{(1-t)^\eta}
$$
  for all  $1/4\leq r\leq t<1$. Joining this with some  elementary estimates, we see that $\eta$ satisfies  (\ref{1023-1}) for all $0\leq r\leq t<1$. This yields
\begin{align}\label{1708}
& \inf\{\beta: \ \beta\  \text{satisfies  (ii) in Lemma \ref{property of doub w}}  \}  \nonumber \\
 \leq & \inf \{\eta: \ \eta \  \text{satisfies (v) in Lemma \ref{property of doub w}}  \}.
\end{align}

By (\ref{1705}), (\ref{1707}) and (\ref{1708}), we get the desired result.
\end{proof}

For $\o \in \hat \d$, denoted by $U(\o)$ the infimum in Lemma \ref{Inf U}.
Generally speaking,  (\ref{1023-1}), (\ref{1023-2}), (\ref{1023-3}) and (\ref{1023-4}) are not true when $\beta=\gamma=\lambda=\eta=U(\o)$.
For example, when $\o(z)=1$, (\ref{1023-2}) and (\ref{1023-3}) do not hold when $\gamma=\lambda=U(\o)=1$.

For a radial weight $\o$, by \cite[Lemma B]{PRS},   $\o\in\check{\mathcal{D}}$ if and only if there exists  $\alpha=\alpha(\o)>0$ such that
\begin{equation}\label{proper L(w)}
\frac{\hat{\o}(t)}{(1-t)^\alpha}\lesssim \frac{\hat{\o}(r)}{(1-r)^\alpha}
\end{equation}
 for all $0\leq r \leq t<1$.
Let $L(\o)$ be the supremum of the set of these parameters  $\alpha$.
There also  exists $\o\in \check{\mathcal{D}}$ such that (\ref{proper L(w)}) does not hold when $\alpha=L(\o)$.
For example, set  $\o(t)=(1-t)\left(2\log\frac{e}{1-t}-1\right)$.  Then  $\hat{\o}(t)= (1-t)^2\log\frac{e}{1-t}$ and $L(\o)=2$
, but (\ref{proper L(w)}) does not hold for $\alpha=2$ and this $\hat{\o}$.

Note that   $\d= \hat{\d} \cap \check{\d}$. Then for any $\o \in \d$, both $U(\o)$ and  $L(\o)$ are well defined. It is also clear that  $0<L(\o)\leq U(\o)<\infty.$

\subsection{Double integral estimates for $B_p(\o)$ spaces}  In this subsection, we give a complete description of radial weights  $\o$ such that double integral estimates for $B_p(\o)$ spaces hold.

For real parameters $p$ and  nonnegative functions  $\o$ on  $\D$,  we will write  $\o_{[p]}(z)=(1-|z|^2)^p \o(z)$ and $dA_p(z)=(1-|z|^2)^p dA(z)$  for convenience.
Now  we give Theorem \ref{1main},  one side of  double integral estimates for $B_p(\o)$ which  always holds. It is worth mentioning that  the weight function $\o$ in the following theorem is not necessarily radial.
\begin{theor} \label{1main}
Suppose  $p>0$,  $\tau>-1$ and  $\o$ is  a nonnegative function on $\D$. Let  $\sigma$ be a real number  such that $\o_{[\sigma]}$ is a weight. Then
{\small
\begin{eqnarray}\label{0213-1}
\int_{\D}\int_{\D}\frac{|f(z)-f(\zeta)|^p}{|1-\overline{\zeta}z|^{4+\tau+\sigma}}\o(\zeta)dA_\tau(z)d A_\sigma(\zeta)
\gtrsim  \ind |f'(z)|^p \o(z) dA_{p-2}(z)
\end{eqnarray}
}
for all $f\in H(\D)$.
\end{theor}
\begin{proof}
For $f\in H(\D)$ and $\zeta\in \D$, by a change of variable $z=\vp_\zeta(u)$, we get
\begin{align}
J_1(\zeta):&=(1-|\zeta|^2)^\sigma\int_{\D}\frac{|f(z)-f(\zeta)|^p}{|1-\overline{\zeta}z|^{4+\tau+\sigma}} d A_\tau(z)  \nonumber\\
&=\frac{1}{(1-|\zeta|^2)^2}\int_\D \frac{ |f\circ\vp_\zeta(u)-f\circ\vp_\zeta(0)|^p }{|1-\ol{\zeta}u|^{\tau-\sigma}}dA_\tau (u).   \label{0213-2}
\end{align}
  From Proposition 4.5 and Lemma 4.30 in \cite{Zhu},
\begin{align}\label{0213-3}
|1-\ol{\zeta}u|\thickapprox |1-\ol{\zeta}\eta|, 1-|u|^2\thickapprox 1-|\eta|^2\thickapprox |1-\ol u \eta|,
\end{align}
 for all $\zeta\in \D$, and all  $u$ and $\eta$ in $\D$ satisfying  $\rho(u,\eta)<1/2$.
 The comparison constants in (\ref{0213-3}) are independent of $\zeta$, $u$ and $\eta$.  These facts allow us to apply   a Littlewood-Paley estimate for Bergman spaces  with nonradial weights $\o_\zeta(u)=(1-|u|^2)^\tau/ {|1-\ol{\zeta}u|^{\tau-\sigma}}$ (see  Theorem \ref{1Bao etac} and its proof in \cite{BWZ}). We get
\begin{align}\label{0213-4}
&\int_\D \frac{ |f\circ\vp_\zeta(u)-f\circ\vp_\zeta(0)|^p }{|1-\ol{\zeta}u|^{\tau-\sigma}}dA_\tau (u) \nonumber \\
\geq &C \int_\D \frac{ |(f\circ\vp_\zeta)^\p (u)|^p }{|1-\ol{\zeta}u|^{\tau-\sigma}}dA_{\tau+p} (u),
\end{align}
where $C$ is a positive constant independent of $\zeta$ and $f$.
Bearing in mind (\ref{0213-2}), (\ref{0213-3}),  (\ref{0213-4}), the change of variable with $u=\vp_\zeta(z)$ and the sub-mean value property for $|f|^p$, we deduce
\begin{align*}
J_1(\zeta)&\gtrsim \frac{1}{(1-|\zeta|^2)^2}\int_\D \frac{ |(f\circ\vp_\zeta)^\p (u)|^p}{ |1-\ol{\zeta}u|^{\tau-\sigma}}dA_{\tau+p} (u)\\
&\thickapprox \int_\D |f^\p(z)|^p\frac{(1-|\zeta|^2)^\sigma }{|1-\ol{\zeta}z|^{\sigma+\tau+4}} dA_{\tau+p}(z) \\
&\gtrsim (1-|\zeta|^2)^{p-4}\int_{\Delta(\zeta, 1/2)} |f^\p(z)|^p dA(z) \\
&\gtrsim (1-|\zeta|^2)^{p-2}|f^\p(\zeta)|^p
\end{align*}
for all $\zeta \in \D$. This implies  (\ref{0213-1}).  The proof is complete.
\end{proof}
The conditions  of $\tau$  and $\sigma$ in Theorem \ref{1main} are only used to ensure the convergence of the integrals in the left-hand of (\ref{0213-1}).

Now we give the other side of double integral estimates for $B_p(\o)$ spaces as follows.

\begin{theor}\label{2main1}
Suppose $p>0$ and $\op$  is a  radial weight. Then the following conditions are equivalent:
\begin{enumerate}[(i)]
  \item   there exist real numbers $\sigma$ and $\tau$ such that
  \begin{eqnarray}\label{18-1}
\int_{\D}\int_{\D}\frac{|f(z)-f(\zeta)|^p}{|1-\overline{\zeta}z|^{4+\tau+\sigma}}\o(\zeta)dA_\tau(z)A_\sigma(\zeta)
\lesssim \ind |f'(z)|^p  \o(z) dA_{p-2}(z)
\end{eqnarray}
for all $f\in H(\D)$;
  \item   $\op\in\hD$.
\end{enumerate}
\end{theor}

To understand well the existence  of parameters $\sigma$ and $\tau$ in Theorem \ref{2main1}, we prove the following result which implies Theorem \ref{2main1}.

\begin{theor}\label{2main}
Suppose $p>0$ and $\o$ is a radial nonnegative function on $\D$. Then the following statements hold:
\begin{itemize}
  \item [(i)] if $\op$  is a  weight and there exist real numbers $\sigma$ and $\tau$ such that
 (\ref{18-1}) holds for all $f\in H(\D)$, then $\op\in\hD$;
  \item [(ii)] if $\op\in\hD$, then (\ref{18-1}) holds for all $f\in H(\D)$ when
  $$\min\{\sigma,\tau\}>p-2,  \,\,\,\tau>\max\{U(\op)-p-1, -1\}.$$
\end{itemize}
\end{theor}
\begin{proof}
(i) Let  $\op$  be  a  weight and there exist real numbers $\sigma$ and $\tau$ such that (\ref{18-1}) holds for all $f\in H(\D)$. This forces  $\tau>-1$
 and $\o_{[\sigma]}$ is a weight.
If $\sigma>\tau$, then
\begin{eqnarray*}
 \frac{(1-|z|)^{\sigma}(1-|\zeta|)^{\sigma}\o(\zeta)}{|1-\overline{\zeta}z|^{4+2\sigma}}
 &\leq& \frac{(1-|z|)^{\tau}(1-|\zeta|)^{\sigma}\o(\zeta)}{|1-\overline{\zeta}z|^{4+\tau+\sigma}} \\
&\leq&  \frac{(1-|z|)^{\tau}(1-|\zeta|)^{\tau}\o(\zeta)}{|1-\overline{\zeta}z|^{4+2\tau}}
\end{eqnarray*}
for all $z$, $\zeta \in \D$.  Thus, without loss of generality,  we can assume that  $\sigma=\tau$. Checking the proof of Theorem \ref{1main}, we get that
\begin{align}
J_1(\zeta)&=(1-|\zeta|^2)^\sigma\int_{\D}\frac{|f(z)-f(\zeta)|^p}{|1-\overline{\zeta}z|^{4+\tau+\sigma}} d A_\tau(z)  \nonumber\\
&\gtrsim \int_\D |f^\p(z)|^p\frac{(1-|\zeta|^2)^\tau dA_{\tau+p}(z)}{|1-\ol{\zeta}z|^{2\tau+4}}   \label{18-2}
\end{align}
for all $\zeta\in \D$. It follows from (\ref{18-1}), (\ref{18-2}) and the Fubini theorem that
\begin{align}
&\int_{\D}|f^\p(z)|^p   \ind  \frac{\o(\zeta) }{|1-\overline{\zeta} z|^{2\tau +4}} dA_\tau(\zeta)dA_{p+\tau}(z)  \nonumber \\
\lesssim & \ind |f^\p(z)|^p  \o(z) dA_{p-2}(z) \label{18-3}
\end{align}
for all $f\in H(\D)$. Take  $f^\p(z)=z^n$ in (\ref{18-3}), where $n=1, 2, \cdots$. Bear in mind  $2\tau+3>0$ and Lemma \ref{well-known estimate}. Then
\begin{eqnarray}
&~& \int_0^1 r^{np+1}(1-r)^{p-2}\o(r)dr \nonumber \\
&\gtrsim&  \int_0^1  t^{np+1} (1-t)^{p+\tau}  \int_0^1 \frac{\o(r)(1-r)^\tau r }{(1-tr)^{2\tau +3}} dr  dt \nonumber \\
&\thickapprox&  \int_0^1  \o(r)(1-r)^\tau   \left(\int_0^1 \frac{t^{np+1} (1-t)^{p+\tau} }{(1-tr)^{2\tau +3}} dt \right) dr. \label{18-4}
\end{eqnarray}
Also,
\begin{eqnarray*}
\int_0^1 \frac{t^{np+1} (1-t)^{p+\tau} }{(1-tr)^{2\tau +3}} dt
&\geq&  \int_{\sqrt{r}}^\frac{1+\sqrt{r}}{2}  \frac{t^{np+1}(1-t)^{p+\tau} }{(1-tr)^{2\tau +3}} dt\\
&\gtrsim& r^\frac{np+1}{2}(1-r)^{p-\tau-2}
\end{eqnarray*}
for  all $r\in [0, 1)$ and for all positive integers $n$. By  this and   (\ref{18-4}), there exists a constant $M\in(1,\infty)$ such that
\begin{equation} \label{18-5}
   \int_0^1  r^\frac{np+1}{2} \op(r) dr  \leq M \int_0^1 r^{np+1}\op(r)dr
\end{equation}
for all positive integers $n$. For $b\in \mathbb{R}$, let $E(b)$ be the integer with $E(b)\leq b <E(b)+1$. Next we use (v) in Lemma \ref{property of doub w} to show $\op\in \hat{\mathcal D}$.

Let $\frac{p+1}{2}\leq x \leq y<\infty$. Write $n_0=E(\frac{2x-1}{p})$ and $k=E(\log_2 \frac{2y}{n_0p})+1$ for convenience.
Then
$$
\frac{n_0p+1}{2}\leq x< \frac{(n_0+1)p+1}{2}, \  \text{and} \ \ y<2^{k-1}n_0p.
 $$
Combining these with (\ref{18-5}), we deduce
\begin{eqnarray}
\int_0^1 r^x\op(r)dr
&\leq& \int_0^1 r^\frac{n_0 p+1}{2}\op(r)dr \nonumber \\
&\leq& M^k  \int_0^1 r^{2^{k-1} n_0 p +1}\op(r)dr \nonumber \\
&\leq& M \left(\frac{2y}{n_0p}\right)^{\log_2 M} \int_0^1 r^y\op(r)dr \nonumber \\
&\leq& C_M  \left(\frac{y}{x}\right)^{\log_2 M}\int_0^1 r^y\op(r)dr,  \label{18-6}
\end{eqnarray}
where  $C_M=M\sup\limits_{n_0\geq 1}(\frac{n_0p+p+1}{n_0p})^{\log_2 M}$ is independent of $x$ and $y$.

Let  $0< x\leq y\leq  \frac{p+1}{2}$.  Clearly,
\begin{align*}
\int_0^1 r^x\op(r)dr
&\thickapprox  \int_0^1 r^y\op(r)dr \\
&\lesssim  \left(\frac{y}{x}\right)^{\log_2M } \int_0^1 r^y\op(r)dr.
\end{align*}

Let  $0<x<\frac{p+1}{2} < y <\infty$. Using (\ref{18-6}), we get
\begin{align*}
\int_0^1 r^x\op(r)dr
&\thickapprox\int_0^1 r^{\frac{p+1}{2}}\op(r)dr \\
&\lesssim C_M \left(\frac{y}{x}\right)^{\log_2M } \int_0^1 r^{y}\op(r)dr.
\end{align*}

Consequently,
$$
\int_0^1 r^x\op(r)dr \lesssim  \left(\frac{y}{x}\right)^{\log_2M }  \int_0^1 r^{y}\op(r)dr
$$
for all $0<x\leq y<\infty$. Hence   Lemma    \ref{property of doub w} yields  $\op\in \hat{\d}$.

(ii)  Suppose $\op\in \hD$,
\begin{equation}\label{parameter}
\min\{\sigma,\tau\}>p-2, \ \text{and}\ \   \tau>\max\{U(\op)-p-1, -1\}.
\end{equation}
Of course, $\op$ is a weight.  Then $\o_{[\sigma]}$ is also a weight.  Next we  show that (\ref{18-1}) holds for all $f\in H(\D)$.

Let $\sigma\geq \tau$.  Following the proof of Theorem \ref{1main} and using Theorem \ref{2Bao etac},
we obtain
\begin{align}
&\int_{\D}\int_{\D}\frac{|f(z)-f(\zeta)|^p}{|1-\overline{\zeta}z|^{4+\tau+\sigma}}\o(\zeta)dA_\tau(z)A_\sigma(\zeta) \nonumber \\
=& \ind  \frac{\o(\zeta)}{(1-|\zeta|^2)^2} dA(\zeta)\int_\D \frac{ |f\circ\vp_\zeta(u)-f\circ\vp_\zeta(0)|^p }{|1-\ol{\zeta}u|^{\tau-\sigma}}dA_\tau (u) \nonumber \\
\leq & 2^{\sigma-\tau} \ind  \frac{\o(\zeta)}{(1-|\zeta|^2)^2} dA(\zeta)\int_\D  |f\circ\vp_\zeta(u)-f\circ\vp_\zeta(0)|^p dA_\tau (u) \nonumber \\
\lesssim &  \ind  \frac{\o(\zeta)}{(1-|\zeta|^2)^2} dA(\zeta)\int_\D  |(f\circ\vp_\zeta(u))'|^p dA_{\tau+p} (u) \nonumber \\
\thickapprox & \int_\D |f^\p(z)|^p\left(\int_\D\frac{ (1-|\zeta|^2)^\tau \o(\zeta)dA(\zeta)}{|1-\ol{\zeta}z|^{2\tau+4}}\right)dA_{\tau+p}(z). \label{18-7}
\end{align}

Let $\sigma< \tau$. For $\zeta\in \D$, write $\eta_\zeta(u)=|1-\ol{\zeta}u|^{\sigma-\tau}(1-|u|^2)^\tau$. Note that $\sigma>p-2>-2$.  Clearly, there exists $s_0\in[-1,0)$ and $t_0>0$ such that if  $s>s_0$ and $t>t_0$, then
$$
s+\tau>-1, \  2+s+t> 0, \ t>\tau, \ \sigma+s>-2.
$$
It follows from Lemma \ref{Forelli-Rudin type  estimates} that
\begin{align*}
\int_\D \frac{\eta_\zeta(u)(1-|u|^2)^sdA(u)}{|1-z\ol{u}|^{2+s+t}}
&=\int_\D \frac{(1-|u|^2)^{s+\tau}dA(u)}{|1-z\ol{u}|^{2+s+t} |1-\ol{\zeta}u|^{\tau-\sigma}}  \\
&\leq C \frac{\eta_\zeta(z)}{(1-|z|^2)^t},
\end{align*}
where $C$ is a positive constant independent of $\zeta$. Using Theorem \ref{2Bao etac} and checking its proof in \cite{BWZ}, we get that there exits another positive constant $C$ independent of $\zeta$ such that
\begin{align*}
&\ \int_\D \frac{ |f\circ\vp_\zeta(u)-f\circ\vp_\zeta(0)|^p }{|1-\ol{\zeta}u|^{\tau-\sigma}}dA_\tau (u)  \\
& \leq C \ind \frac{|(f\circ\vp_\zeta(u))'|^p (1-|u|^2)^p } {|1-\ol{\zeta}u|^{\tau-\sigma}}   dA_\tau (u)  \\
&= C \int_\D |f^\p(z)|^p\frac{(1-|\zeta|^2)^{\sigma+2} }{|1-\ol{\zeta}z|^{\sigma+\tau+4}} dA_{\tau+p}(z).
\end{align*}
This yields
\begin{align}
&\int_{\D}\int_{\D}\frac{|f(z)-f(\zeta)|^p}{|1-\overline{\zeta}z|^{4+\tau+\sigma}}\o(\zeta)dA_\tau(z)A_\sigma(\zeta) \nonumber \\
\lesssim & \int_\D |f^\p(z)|^p\left(\int_\D\frac{ (1-|\zeta|^2)^\sigma  \o(\zeta)dA(\zeta)}{|1-\ol{\zeta}z|^{\sigma+\tau+4}}\right)dA_{\tau+p}(z). \label{18-8}
\end{align}

Write  $x=\min\{\sigma,\tau\}$.    It follows from (\ref{18-7}) and (\ref{18-8}) that
\begin{align}
&\int_{\D}\int_{\D}\frac{|f(z)-f(\zeta)|^p}{|1-\overline{\zeta}z|^{4+\tau+\sigma}}\o(\zeta)dA_\tau(z)A_\sigma(\zeta)   \nonumber\\
\lesssim & \int_\D |f^\p(z)|^p\left(\int_\D\frac{ (1-|\zeta|^2)^x\o(\zeta)dA(\zeta)}{|1-\ol{\zeta}z|^{\tau+x+4}}\right)dA_{\tau+p}(z)\nonumber\\
:= &\int_\D |f^\p(z)|^p d\mu(z).  \label{18-9}
\end{align}
By (\ref{parameter}), $\tau+x+3>p>0$.
For any $a\in\D\backslash \{0\}$, Lemma \ref{well-known estimate} yields
\begin{align}
\mu(S(a))=&\int_{S(a)}\left(\int_\D\frac{ (1-|\zeta|^2)^x\o(\zeta)dA(\zeta)}{|1-\ol{\zeta}z|^{\tau+x+4}}\right)dA_{\tau+p}(z) \nonumber \\
\thickapprox &(1-|a|)\int_{|a|}^1 (1-r)^{\tau+p}\int_0^{|a|}\frac{(1-t)^x \o(t)dt}{(1-rt)^{\tau+x+3}} dr \nonumber \\
&+ (1-|a|)\int_{|a|}^1 (1-r)^{\tau+p}\int_{|a|}^1 \frac{(1-t)^x \o(t)dt}{(1-rt)^{\tau+x+3}} dr \nonumber \\
: = &(1-|a|) (I_1(a)+I_2(a)). \label{bu 1}
\end{align}
If $|a|<r<1$ and $0<t<|a|$, then
$$1>\frac{1-t}{1-rt}>\frac{1-|a|}{1-r|a|}>\frac{1-|a|}{1-|a|^2}>\frac{1}{2}.$$
Note that  $\tau>U(\op)-p-1$.  By Lemma \ref{property of doub w}, we have
\begin{eqnarray}
I_1(a)&\thickapprox & \int_{|a|}^1 (1-r)^{p+\tau} \int_0^{|a|}  \frac{\o(t)(1-t)^{p-2}}{(1-t)^{\tau+p+1}} dt   dr \nonumber \\
&\lesssim &   \frac{\widehat{\op}(a)}{(1-|a|)^{\tau+p+1}} \int_{|a|}^1 (1-r)^{p+\tau}dr \nonumber \\
&\thickapprox& \widehat{\op}(a). \label{bu 2}
\end{eqnarray}
Due to   $x>p-2$ and $\tau+p>-1$,
$$\int_{0}^1 \frac{(1-r)^{\tau+p}}{(1-rt)^{\tau+x+3}}  dr\thickapprox \frac{1}{ (1-t)^{x+2-p}}$$
for all $t\in (0, 1)$. By this and the Fubini theorem, one gets
\begin{align*}
I_2(a)=&\int_{|a|}^1 \o(t)(1-t)^{x}\int_{|a|}^1 \frac{(1-r)^{\tau+p}}{(1-rt)^{\tau+x+3}}  dr dt\\
\lesssim& \int_{|a|}^1 \o(t)(1-t)^{p-2}\int_{0}^1 \frac{(1-r)^{\tau+p}(1-t)^{x+2-p}}{(1-rt)^{\tau+x+3}}  dr dt\\
\thickapprox& \widehat{\op}(a).
\end{align*}
Note that
$$
\int_{S(a)}\o_{[p-2]}(z)dA(z)\thickapprox (1-|a|)\widehat{\op}(a)
$$
for all $a\in\D\backslash \{0\}$.
Consequently,
$$
\sup_{a\in\D\backslash \{0\}} \frac{\mu(S(a))}{\int_{S(a)}\o_{[p-2]}(z)dA(z)}<\infty.
$$
It follows from  (\ref{18-9}) and Theorem \ref{CM for Apw} that
\begin{align*}
&\int_{\D}\int_{\D}\frac{|f(z)-f(\zeta)|^p}{|1-\overline{\zeta}z|^{4+\tau+\sigma}}\o(\zeta)dA_\tau(z)A_\sigma(\zeta)   \nonumber\\
\lesssim &
\int_\D |f^\p(z)|^p\left(\int_\D\frac{ (1-|\zeta|^2)^x\o(\zeta)dA(\zeta)}{|1-\ol{\zeta}z|^{\tau+x+4}}\right)dA_{\tau+p}(z)\\
\lesssim &  \int_\D |f^\p(z)|^p\op(z)dA(z)
\end{align*}
for all $f\in H(\D)$.
The proof is complete.
\end{proof}

For a nonnegative function $\eta$ on [0, 1), we say that $\eta$ is essentially decreasing on  [0, 1) if $\eta(t_1)\gtrsim \eta (t_2)$ for all $0\leq t_1<t_2<1$. Also, $\eta$ is said to be essentially increasing on  [0, 1) if $\eta(t_1)\lesssim  \eta (t_2)$ for all $0\leq t_1<t_2<1$.

Note that $\d \subsetneqq \hD$. For $0<p<\infty$ and a radial nonnegative function $\o$ on $\D$, if we assume  $\op\in\mathcal \d$, then the range of parameters $\sigma$ or  $\tau$ in (ii) of Theorem \ref{2main} should be larger. Because of this observation, we give  the following result.

\begin{prop}\label{R sigma tau}
Suppose $p>0$ and $\o$ is a radial nonnegative function on $\D$. If $\op\in\d$, then  (\ref{18-1}) holds for all $f\in H(\D)$ when
\begin{equation}\label{07231}
\tau>-1, \sigma>-2, \min\{\sigma,\tau\}> p-2-L(\op), \tau> U(\op)-p-1,
\end{equation}
 and $\o_{[\sigma]}$ is a  weight.
\end{prop}
\begin{proof}
Checking  the proof of (ii) of Theorem \ref{2main} step by step, we see that  (\ref{bu 1}) and (\ref{bu 2}) hold. Also, $I_2(a)\lesssim  \widehat{\op}(a)$ for all $a\in\D\backslash \{0\}$  when $\min\{\sigma,\tau\}> p-2$.
Hence, it suffices to prove  $I_2(a)\lesssim  \widehat{\op}(a)$ for all $a\in\D\backslash \{0\}$  when $p-2-L(\op)<\min\{\sigma,\tau\}\leq p-2$.
Recall that  $x=\min\{\sigma,\tau\}$. Write
\begin{align}
I_2(a)
&=\int_{|a|}^1 \o(t)\int_{|a|}^t \frac{(1-r)^{\tau+p}(1-t)^{x}}{(1-rt)^{\tau+x+3}}  dr dt \nonumber \\
&\ \ \ \ + \int_{|a|}^1 \o(t)\int_t^1 \frac{(1-r)^{\tau+p}(1-t)^{x}}{(1-rt)^{\tau+x+3}}  dr dt \nonumber \\
&:=I_{21}(a)+I_{22}(a). \label{B1}
\end{align}
Clearly,
\begin{equation}
I_{22}(a)
\lesssim \int_{|a|}^1 \o(t)(1-t)^{-\tau-3}\int_t^1 (1-r)^{\tau+p} dr dt
\thickapprox \widehat{\op}(a).  \label{B2}
\end{equation}
The Fubini theorem yields
\begin{align*}
I_{21}(a)
&=\int_{|a|}^1 \int_r^1  \frac{(1-r)^{\tau+p}(1-t)^{x}\o(t)}{(1-rt)^{\tau+x+3}}  dt dr  \\
&\lesssim \int_{|a|}^1 \int_r^1  (1-r)^{p-x-3}(1-t)^{x}\o(t)  dt dr.
\end{align*}
If  $x=p-2$, due to  $\op\in \check{\mathcal{D}}$, it follows from (\ref{proper L(w)}) that
\begin{align} \label{B3}
I_{21}(a)\lesssim \int_{|a|}^1 (1-r)^{-1}\widehat{\op}(r)dr\lesssim \widehat{\op}(a).
\end{align}
If $p-2-L(\op)<x<p-2$, we can choose  a small enough positive number $\varepsilon$ such that
$p-2-L(\op)+\varepsilon<x$ and
$\frac{\widehat{\op}(t)}{(1-t)^{L(\op)-\varepsilon}}$
is essentially decreasing. An integration by parts gives that
\begin{align*}
&\int_{|a|}^1 \o(t)(1-t)^x dt\\
=&(1-|a|)^{x-p+2}\widehat{\op}(a)-(x-p+2)\int_{|a|}^1 \widehat{\op}(t)(1-t)^{x-p+1}dt.
\end{align*}
Consequently,
\begin{align}
I_{21}(a)
&\lesssim \int_{|a|}^1 \int_r^1  (1-r)^{p-x-3}(1-t)^{x}\o(t)  dt dr \nonumber \\
&\thickapprox\int_{|a|}^1 \o(t)(1-t)^x dt\int_{|a|}^t(1-r)^{p-x-3}dr \nonumber  \\
&\lesssim  (1-|a|)^{p-x-2} \int_{|a|}^1 \o(t)(1-t)^x dt \nonumber  \\
&\lesssim  \widehat{\op}(a)+(1-|a|)^{p-x-2} \int_{|a|}^1 \frac{\widehat{\op}(t)(1-t)^{L(\op)-\varepsilon}}{(1-t)^{L(\op)-\varepsilon}(1-t)^{p-x-1}}dt \nonumber  \\
&\lesssim \widehat{\op}(a). \label{B4}
\end{align}
Joining (\ref{B1}), (\ref{B2}), (\ref{B3}) and (\ref{B4}), we see that $I_2(a)\lesssim  \widehat{\op}(a)$ for all $a\in\D\backslash \{0\}$  when $p-2-L(\op)<\min\{\sigma,\tau\}\leq p-2$. The proof is finished.
\end{proof}

Theorem \ref{1main} and Theorem  \ref{2main1} give  complete descriptions  of nonnegative radial weight functions $\omega$ such that the double integral estimates  for $B_p(\o)$  hold. Theorem \ref{2main}
and Proposition \ref{R sigma tau} explain further   parameters $\sigma$ and $\tau$ in (\ref{18-1}).

As stated in Section 1 (cf. \cite[Theorem 2.2]{BP}), when $p>1$ and $\alpha\leq 1/2$, there is a  double integral characterization for $B_p(\alpha)$. Theorem \ref{Bpa} below covers this case exactly.
In fact, Theorem \ref{Bpa}  gives that this  double integral characterization for $B_p(\alpha)$ exists for all possible   $p$ and $\alpha$.

\begin{theor}\label{Bpa}
Suppose
$p>\max\{0,2\alpha\}$, $\beta>\max\{-1,-2\alpha-1\}$ and
$$\tau> \max\{-1, -2+2\alpha,-1-2\alpha\}.$$
Then
\begin{align*}
&\int_{\D}\int_{\D}\frac{|f(z)-f(\zeta)|^p}{|1-\overline{\zeta}z|^{3+\beta+\tau+2\alpha}} (1-|\zeta|^2)^{\beta} (1-|z|^2)^\tau dA(z)d A(\zeta)\\
\thickapprox & \ind |f'(z)|^p  (1-|z|^2)^{p-1-2\alpha} dA(z)
\end{align*}
for all $f\in H(\D)$.
\end{theor}
\begin{proof}
Take  $\o(z)=(1-|z|^2)^{1-2\alpha}$ and  $\sigma=\beta-1+2\alpha$.
Then $U(\op)=L(\op)=p-2\alpha$.
Applying   Theorem \ref{1main} and Proposition \ref{R sigma tau} to this $\o$,   we obtain
\begin{align*}
&\int_{\D}\int_{\D}\frac{|f(z)-f(\zeta)|^p}{|1-\overline{\zeta}z|^{4+\tau+\sigma}} (1-|\zeta|^2)^{1-2\alpha+\sigma} (1-|z|^2)^\tau dA(z)d A(\zeta)\\
\thickapprox & \ind |f'(z)|^p  (1-|z|^2)^{p-1-2\alpha} dA(z)
\end{align*}
for all $f\in H(\D)$.  The  desired result follows.
\end{proof}

\section{Geometric characterizations of functions in $B_p(\o)$ spaces}

In this section, based on   double integral estimates for $B_p(\o)$ spaces, we give some  geometric descriptions  of functions in $B_p(\o)$ spaces when $0<p<\infty$ and  $\op\in\hD$.

For any $z$, $\zeta$ in $\D$, denote by $\beta(z, \zeta)$ the distance between $z$ and $\zeta$ in the Bergman metric. It is well known that
$$
\beta(z, \zeta)=\frac{1}{2}\log \frac{1+\rho(z, \zeta)}{1-\rho(z, \zeta)}.
$$
Suppose $z\in \D$ and $r>0$. Let $D(z, r)=\{\zeta\in \D:  \beta(z, \zeta)<r\}$  and let $|D(z, r)|$ be the area of $D(z, r)$ with respect to the  measure $dA$. For $f\in H(\D)$, denote by  $|f(D(z, r))|$  the area of the image of $D(z, r)$ under $f$. Recall that
$$
O_r(f)(z)=\sup \{|f(z)-f(\zeta)|: \ \ \zeta\in  D(z, r)\}
$$
is  the oscillation of $f$ at $z$ in the Bergman metric.  The mean oscillation of $f$ at $z$ in the Bergman metric is
$$
MO_r(f)(z)=\frac{1}{|D(z, r)|} \int_{D(z, r)} |f(\zeta)-\widehat{f_r}(z)|dA(\zeta),
$$
where $\widehat{f_r}$ is the averaging function given by
$$
\widehat{f_r}(z)= \frac{1}{|D(z, r)|} \int_{D(z, r)} f(u)dA(u).
$$

Applying Theorem \ref{2main}, we obtain  the following conclusion which generalizes Theorem B in \cite{Zhu1} from the Besov space $B_p$, $p>1$, to the  Besov type space $B_p(\o)$  for   $0<p<\infty$ and  $\op\in\hD$.
The proof of Theorem B in \cite{Zhu1} used the H\"older inequality. Theorem \ref{8main1} below includes the case of $0<p<1$ for which the H\"older inequality is invalid.

\begin{theor}\label{8main1}
Suppose $f\in H(\D)$,  $0<r<\infty$, $0<p<\infty$, and $\op\in\hD$. Then the following conditions are equivalent:
\begin{enumerate}[(i)]
\item $f\in B_p(\o)$;
\item the function  $z\mapsto |f(D(z, r))|^{1/2}$  belongs to   $L^p(\D, wd\lambda)$;
\item $O_r(f) \in L^p(\D, wd\lambda)$;
\item $MO_r(f) \in L^p(\D, wd\lambda)$.
\end{enumerate}
\end{theor}
\begin{proof}
$(i)\Rightarrow (ii)$.  Let $f\in B_p(\o)$.  By the proof of Theorem \ref{2main}, there exist real numbers $x$ and $\tau$ such that
\begin{align}
&\int_\D |f^\p(z)|^p\left(\int_\D\frac{ (1-|\zeta|^2)^x\o(\zeta)dA(\zeta)}{|1-\ol{\zeta}z|^{\tau+x+4}}\right)dA_{\tau+p}(z) \nonumber \\
\lesssim &  \int_\D |f^\p(z)|^p\op(z)dA(z)<\infty.  \label{7111}
\end{align}
Note that for $z\in \D$ the integral $\int_{D(z, r)}|f'(\zeta)|^2 dA(\zeta)$, familiar from the theory of the classical Dirichlet space, is equal to the area of
$f(D(z, r))$ counting multiplicities. The quantity $|f(D(z, r))|$  disregards  multiplicities. Hence
\begin{align}\label{7112}
|f(D(z, r))|\leq \int_{D(z, r)}|f'(\zeta)|^2 dA(\zeta).
\end{align}
It is well known that  $|D(z, r)|\thickapprox (1-|z|^2)^2$ for all $z\in \D$. For $\zeta \in D(z, r)$,  then    $D(\zeta, r)\subseteq D(z, 2r)$,   and
$1-|z|\thickapprox 1-|\zeta| \thickapprox |1-\overline{z} \zeta|$. Hence,
\begin{align}
|f'(\zeta)|^p & \lesssim \frac{1}{|D(\zeta, r)|} \int_{D(\zeta, r)} |f'(u)|^p  dA(u) \nonumber \\
 & \lesssim \frac{1}{|D(z, r)|} \int_{D(z, 2r)} |f'(u)|^p  dA(u) \label{7113}
\end{align}
for all $\zeta \in D(z, r)$.  Using (\ref{7113}), we  deduce
\begin{align*}
&\left[\int_{D(z, r)}|f'(\zeta)|^2 dA(\zeta)\right]^{p/2}\\
  \leq  & |D(z, r)|^{p/2} \sup \{|f'(\zeta)|^p: \ \ \zeta \in D(z, r)\} \nonumber \\
 \lesssim   & (1-|z|)^{p-2}  \int_{D(z, 2r)} |f'(u)|^p  dA(u) \nonumber \\
 \thickapprox & \int_{D(z, 2r)} \frac {|f'(u)|^p (1-|u|^2)^{\tau+p} (1-|z|^2)^{x+2} }{|1-\overline{z} u|^{4+\tau+x}}dA(u)  \nonumber
\end{align*}
for all $z\in \D$. By  this estimate, the Fubini theorem and  (\ref{7111}), we get
\begin{align*}
&\ind \left[\int_{D(z, r)}|f'(\zeta)|^2 dA(\zeta)\right]^{p/2} \o(z) d\lambda(z)\\
\lesssim   &  \ind \ind  \frac {|f'(u)|^p (1-|u|^2)^{\tau+p} (1-|z|^2)^{x+2} }{|1-\overline{z} u|^{4+\tau+x}}dA(u)  \o(z) d\lambda(z)\\
\thickapprox & \ind   |f'(u)|^p (1-|u|^2)^{\tau+p} \ind    \frac { (1-|z|^2)^{x} \o(z) }{|1-\overline{z} u|^{4+\tau+x}}   dA(z)dA(u)\\
< & \infty.
\end{align*}
Combining this with (\ref{7112}), we get that  the function  $z\mapsto |f(D(z, r))|^{1/2}$  belongs to   $L^p(\D, wd\lambda)$.

$(ii)\Rightarrow (i)$. It follows from \cite[p. 322]{Ax} that $(1-|z|^2)|f'(z)|\lesssim |f(D(z, r))|^{1/2}$ for all $z\in \D$. Thus this  implication is true.

$(i)\Rightarrow (iii)$.  Let $f\in B_p(\o)$. By Theorem \ref{2main},  there exist real numbers $\sigma$ and $\tau$ such that
\begin{eqnarray}\label{710-1}
\int_{\D}\int_{\D}\frac{|f(z)-f(\zeta)|^p}{|1-\overline{\zeta}z|^{4+\tau+\sigma}}\o(\zeta)dA_\tau(z)A_\sigma(\zeta)
\lesssim \ind |f'(z)|^p  \o(z) dA_{p-2}(z)
\end{eqnarray}
For $z$ and $\zeta$ in $\D$, the subharmonicity of the function $u \mapsto |f(u)-f(z)|^p$ on $\D$ yields
$$
|f(\zeta)-f(z)|^p \leq \frac{C}{|D(\zeta, r)|} \int_{D(\zeta, r)} |f(u)-f(z)|^p dA(u),
$$
where $C$ is a positive constant depending only on $r$ (cf. \cite[Proposition 4.13]{Zhu}).       If  $\zeta\in  D(z, r)$,  then $D(\zeta, r)\subseteq D(z, 2r)$,   and
$1-|z|\thickapprox 1-|\zeta| \thickapprox |1-\overline{z} \zeta|$.  Consequently,
\begin{align*}
[O_r(f)]^p(z)&\lesssim \frac{1}{|D(z, 2r)|} \int_{D(z, 2r)} |f(u)-f(z)|^p dA(u) \\
& \thickapprox  \int_{D(z, 2r)} |f(u)-f(z)|^p \frac{(1-|z|^2)^{2+\sigma}(1-|u|^2)^{\tau}}{|1-z\overline{u}|^{4+\tau+\sigma}}dA(u) \\
& \lesssim \int_\D  |f(u)-f(z)|^p \frac{(1-|z|^2)^{2+\sigma}(1-|u|^2)^{\tau}}{|1-z\overline{u}|^{4+\tau+\sigma}}dA(u) .
\end{align*}
Combining this with (\ref{710-1}), we get  $O_r(f) \in L^p(\D, wd\lambda)$.

$(iii)\Rightarrow (iv)$. Let $O_r(f) \in L^p(\D, wd\lambda)$.  For $z\in \D$, a direct calculation gives
\begin{align*}
MO_r(f)(z)& \leq \frac{1}{|D(z, r)|^2} \int_{D(z, r)}\int_{D(z, r)} |f(\zeta)-f(u)|dA(\zeta)dA(u)\\
&\leq \frac{2}{|D(z, r)|} \int_{D(z, r)} |f(\zeta)-f(z)|dA(\zeta)\\
&\leq 2 O_r(f)(z).
\end{align*}
Hence $MO_r(f) \in L^p(\D, wd\lambda)$.

$(iv)\Rightarrow (i)$.  By \cite[p. 330]{Zhu1},  $(1-|z|^2)|f'(z)|\lesssim MO_r(f)(z)$ for all $z\in \D$. Then the desired result holds.  We finish the proof.
\end{proof}

\section{Hankel type operators related to  $B_p(\o)$ spaces with radial B\'ekoll\'e-Bonami weights}

In this section, applying double integral estimates for Besov type spaces $B_p(\o)$, we characterize the boundedness of  Hankel type operators related to  $B_p(\o)$  with radial B\'ekoll\'e-Bonami weights. This result is  new for $B_p(\alpha)$ when  $1<p<\infty$ and $1/2<\alpha <p/2$. Our result on $B_2(\o)$ also improves  a previous result from the literature. The compactness of these Hankel type operators  is  also considered.

For $p>1$ and $s>-1$, the B\'ekoll\'e-Bonami  class $\B_{p, s}$ consists of  nonnegative and  integrable  functions  $\eta$ on $\D$ with the property  that there exists a positive constant $C$ satisfying
{\small
$$
\left(\int_{S(a)}\eta(z)  dA(z) \right) \left(\int_{S(a)} \left(\frac{\eta(z)}{(1-|z|^2)^s}\right) ^{-\frac{p'}{p}} dA_s(z) \right)^{\frac{p}{p'}}
\leq C (A_s (S(a)))^p
$$}
for all Carleson boxes $S(a)$.  Here and in what follows,   $p'$ is the real number with   $1/p+1/p'=1$.   D. B\'ekoll\'e and A. Bonami \cite{BB} proved that $\eta\in \B_{p,s}$
if and only if the Bergman projection
$$
P_s f(z)= (s +1)\ind \frac{f(\zeta)}{(1-\overline{\zeta}z)^{s+2}} dA_s (\zeta)
$$
is bounded from $L^p(\D, \eta dA)$  to $A^p_\eta$. Note that weights in $\B_{p,s}$ are  not necessarily radial.  From Proposition 6 in \cite{PPR}, any radial weight in $\B_{p,s}$  belongs to $\d$. See \cite{Bek, BB} for these weights.

For $s>-1$, by the Bergman projection $P_s$, it is  possible
to define a (small) Hankel type operator $h_{s, f}$ on $\mathcal P$ by
$$
h_{s, f} g=\overline{P_{s}(f\overline{g})},  \ \ \  g\in \mathcal P.
$$
We refer to \cite{Po, Zhu} for more results on Hankel type operators.

Let  $p>0$ and let $\o_{[p-2]}$ be a weight. Denote by $S_p(\o)$ the Sobolev type space of smooth functions $u: \D\rightarrow \C$ such that
$$
\|u\|_{S_p(\o)}=|u(0)|+ \left(\int_{\mathbb{D}} |\nabla u(z)|^p (1-|z|^2)^{p-2}\omega(z)dA(z)\right)^{1/p}<\infty.
$$
Also, if  $\o(z)=(1-|z|^2)^{1-2\alpha}$, then we write $S_p(\o)$ as $S_p(\alpha)$.
Clearly, $B_p(\o)$ is a subset of  all analytic functions in  $S_p(\o)$.
 Denote by $\mathcal{P}$ the class of  polynomials on $\D$. If $w$ is radial, it is well known that  $\mathcal{P}$ is  dense in $B_p(\o)$ (cf. \cite{Mer}).
For $s>-1$, consider  the operator $\widetilde{P}_{s}$ given by
 $$
 \widetilde{P}_{s}u(z)=u(0)+\ind \frac{\partial u}{\partial w}(w) \frac{1-(1-\overline{w}z)^{1+s}}{\overline{w}(1-\overline{w}z)^{1+s}}(1-|w|^2)^s dA(w).
 $$
 Then one can define a  (small) Hankel type operator on $\mathcal P$ with certain  symbol $f$  by
$$
h_f^{s} g=\overline{\widetilde{P}_{s}(f\overline{g})},  \ \ \  g\in \mathcal P.
$$
 As explained by D. Blasi and J. Pau \cite[p. 402]{BP}, $\widetilde{P}_{1-2\alpha}$ defines a projection from $S_p(\alpha)$ to $B_p(\alpha)$.
 For $f$ analytic in $\D$,
the boundedness of $h_f^{1-2\alpha}$ from $B_p(\alpha)$ to $S_p(\alpha)$  was also  studied in  \cite{BP}. Considering  the same topic  with weights  $\o$, one should define a Hankel type operator by the  projection from
$S_p(\o)$ to $B_p(\o)$. From this way, the results  in Section 2 will not be used in the study. Our  purpose of this section is to present how to apply double integral estimates for  $B_p(\o)$ obtained in Section 2. Thus  we still focus on the operator $h_f^{s}$ in this  section. In particular, some special cases of our main result  can complete or improve some results before.

\subsection{The boundedness of Hankel type operators  from  $B_p(\o)$  to $S_p(\o)$}  In this subsection, we investigate the boundedness of $h^s_f:  B_p(\o) \to S_p(\o)$ with analytic symbol $f$.

We first recall Proposition 5 in \cite{PRS} as follows.
\begin{otherp}\label{w yiwan}
Let $0<p<\infty$, $\o\in \hD$ and write $\tilde{w}(r)=\hat{\o}(r)/(1-r)$ for all $0\leq r<1$.  Then $\o \in \check{\d}$ if and only if $\|f\|_{A^p_\o}\thickapprox \|f\|_{A^p_{\tilde{\o}}}$ for all $f\in H(\D)$.
\end{otherp}

\noindent{\bf  Remark }.
In fact, by the proof of the  proposition above in \cite{PRS},  if  $\o \in \d$,  then $\tilde{\o}\in\mathcal{R}$. Here $\mathcal{R}$ is the regular class of weights consisting  of all $\o\in\d$ such that
$$\hat{\o}(r)\thickapprox (1-r)\o(r), \,\,\,r\in(0,1).$$
Also,   if  $\o \in \d$, then $\hat{\tilde{\o}}(r)\thickapprox \hat{\o}(r)$ for all $0\leq r<1$. Then it is also clear that
$U(\o)=U(\tilde  \o)$  and $L(\o)=L(\tilde  \o)$.

We also need the following result  (cf. \cite[Lemma 3.1]{Pe}).
\begin{otherl}\label{test functions}
Let $0<p<\infty$ and $\o \in \hD$. Then there is real number $\lambda_0=\lambda_0(\o)$ such that for any $\lambda \geq \lambda_0$ and each $a\in \D$ the function $F_{a, p}(z)=\left(\frac{1-|a|^2}{1-\overline{a}z}\right)^{(\lambda+1)/p}$
is analytic in $\D$ and satisfies
$$
|F_{a, p}(z)| \thickapprox 1, \ \ z\in S(a), \ \ a\in \D,
$$
and
$$
\|F_{a, p}\|^p_{A^p_\o}\thickapprox \int_{S(a)} \o dA, \ \ a\in \D.
$$
\end{otherl}

Now  we give the following lemma.

\begin{limma}\label{1021-1}
Suppose  $1<p<\infty$, $-1<s<\infty$,  and  $\eta$ is a radial weight in  $\mathcal{B}_{p, s}$. Let  $f\in A_{\eta}^p$ and let $h_{s,f}$ be  a bounded operator from $B_p(\eta_{[2-p]})$ to $L_\eta^p$.
Then
$$
\sup_{a\in\D} (1-|a|^2)|f(a)|<\infty.
$$
\end{limma}
\begin{proof}
Bear in mind that any radial weight in $\B_{p,s}$  belongs to $\d$.
By  Proposition \ref{w yiwan},  we have
\begin{align}
\|h_{s,f}g\|_{L^p(\D, \eta dA)}
&=\|P_s(f\overline{g})\|_{L^p(\D, \eta dA)} \nonumber\\
&=\|P_s(f\overline{g})\|_{A_\eta^p} \nonumber\\
&\thickapprox \|P_s(f\overline{g})\|_{A_{\tilde{\eta}}^p}
=\|h_{s,f}g\|_{L^p(\D,  \tilde{\eta}dA)} \label{19-2}
\end{align}
for all  $g\in B_p(\eta_{[2-p]})$.
For any $a\in \D$, set
$$
g_a(z)=\frac{z^{n+1}}{(1-\overline{a}z)^{n+1}}, \ \ \ z\in \D,
$$
where $n$ is a positive integer.
By Lemma \ref{test functions},  if $n$ is large enough,
\begin{align}
\|g_a\|_{B_p(\eta_{[2-p]})}\lesssim \frac{(1-|a|)^\frac{1}{p}\hat{\eta}(a)^\frac{1}{p}}{(1-|a|)^{n+2}}  \label{19-1}
\end{align}
for all $a\in \D$. Note that   $h_{s,f}$ is  bounded from $B_p(\eta_{[2-p]})$ to $L_\eta^p$, and $P_sf=\overline{h_{s,f}1}$ . Hence  $P_sf$  belongs to  $A^p_\eta$.   Then  $P_sf=f$. For any fixed $\lambda>-1$, we obtain
\begin{align*}
&f^{(n+1)}(a)\\
=&(s+1)\left( \int_\D  \frac{f(z)(1-|z|^2)^sdA(z)}{(1-a\overline{z})^{s+2}} \right)^{(n+1)}\\
=&C(n,s)(s+1)\int_\D \frac{\overline{z}^{n+1} f(z)}{(1-a\overline{z})^{n+1}}
  \frac{1}{(1-a\overline{z})^{2+s}} (1-|z|^2)^sdA(z)\\
=& C(n,s,\lambda) (s+1)\int_\D \frac{\overline{z}^{n+1} f(z)}{(1-a\overline{z})^{n+1}}
    \left(\int_\D \frac{dA_\lambda(w)}{(1-w\overline{z})^{2+s}(1-a\overline{w})^{\lambda+2}} \right)
   dA_s(z)\\
=&C(n,s,\lambda)  \int_\D \left((s+1)\int_\D \frac{\overline{z}^{n+1} f(z)}{(1-a\overline{z})^{n+1}}\frac{dA_s(z)}{(1-w\overline{z})^{2+s}}\right)
 \frac{dA_\lambda(w)}{(1-a\overline{w})^{\lambda+2}}\\
=&C(n,s,\lambda) \int_\D \overline{h_{s,f}(g_a)(w)}\frac{dA_\lambda(w)}{(1-a\overline{w})^{\lambda+2}},
\end{align*}
where $C(n,s)$ is a constant depending on $n$ and $s$,  and $C(n,s,\lambda)$ is a constant depending on $n$, $s$ and $\lambda$.
It follows from   H\"older's inequality, the boundedness  of $h_{s, f}$,  (\ref{19-2}) and  (\ref{19-1}) that
{\small
\begin{align}
&(1-|a|^2)^{n+2}|f^{(n+1)}(a)|  \nonumber\\
\lesssim &(1-|a|^2)^{n+2}\|h_{s,f}g_a\|_{L^p(\D,  \tilde{\eta}dA)}
 \left(\int_\D \frac{(1-|w|^2)^{{p'}\lambda}d A(w)}
{\tilde{\eta}(w)^\frac{{p'}}{p}|1-a\overline{w}|^{(\lambda+2){p'}}}\right)^\frac{1}{{p'}} \nonumber\\
\lesssim& \|h_{s,f}\|_{B_p(\eta_{[2-p]})\to L^p(\D, \eta dA)}(1-|a|)^{\frac{1}{p}}{\hat{\eta}}(a)^\frac{1}{p}
\left(\int_0^1  \frac{(1-r)^{{p'}\lambda}d r}
{\tilde{\eta}(r)^\frac{{p'}}{p}(1-|a|r)^{(\lambda+2){p'}-1}}\right)^\frac{1}{{p'}}  \label{1020-1}
\end{align} }
for all $a\in \D$.
Since $\tilde{\eta}\in\mathcal{R}$, there exist $-1<\alpha,\beta<\infty$ such that
$\frac{\tilde{\eta}(t)}{(1-t)^\alpha}$ and $\frac{\tilde{\eta}(t)}{(1-t)^\beta}$
are essentially  increasing and essentially  decreasing respectively (cf. \cite[p.  11]{PR2014}).
Note that $-\frac{p'\beta}{p}-2p'+2=(\beta+2)(1-p')<0$.
If $\lambda$ is large enough, then
\begin{align*}
&\int_0^{|a|}  \frac{(1-r)^{{p'}\lambda}d r}{{\tilde{\eta}}(r)^{p'/p} (1-|a|r)^{(\lambda+2){p'}-1}}\\
\lesssim& \left(\frac{(1-|a|)^\beta}{{\tilde{\eta}}(|a|)}\right)^{p'/p}
  \int_0^{|a|}  \frac{(1-r)^{{p'}\lambda-\frac{{p'}\beta}{p}}d r}{(1-|a|r)^{(\lambda+2){p'}-1}}   \\
\lesssim& \left(\frac{(1-|a|)^\beta}{{\tilde{\eta}}(|a|)}\right)^{p'/p}
  \int_0^{|a|}  (1-r)^{-\frac{{p'}\beta}{p}-2{p'}+1}d r \\
\lesssim& \frac{(1-|a|)^{-2{p'}+2}}{{\tilde{\eta}}(a)^{p'/p}},
\end{align*}
and
\begin{align*}
\int_{|a|}^1  \frac{(1-r)^{{p'}\lambda}d r}{{\tilde{\eta}}(r)^{p'/p} (1-|a|r)^{(\lambda+2){p'}-1}}
\leq& \left(\frac{(1-|a|)^\alpha}{{\tilde{\eta}}(|a|)}\right)^{p'/p}
  \int_{|a|}^1  \frac{(1-r)^{{p'}\lambda-\frac{{p'}\alpha}{p}}d r}{(1-|a|r)^{(\lambda+2){p'}-1}}   \\
\lesssim& \frac{(1-|a|)^{-2{p'}+2}}{{\tilde{\eta}}(a)^{p'/p}}
\end{align*}
for all $a\in \D$.
Combining this with   (\ref{1020-1}), we get
$$
\sup_{a\in\D} (1-|a|^2)^{n+2}|f^{(n+1)}(a)|<\infty,
$$
which is equivalent to
$$
\sup_{a\in\D}(1-|a|^2)|f(a)|<\infty.
$$
The proof is complete.
\end{proof}

The following theorem is the main result in this section and  the double integral estimates for Besov type spaces will be used in its proof.

\begin{theor}\label{3main}
Suppose $1<p<\infty$, $-1<s<\infty$, $f\in H(\D)$ and $\eta\in \B_{p, s}$ is a radial weight. Let $\eta,p,s$ satisfy  further one of the following conditions:
\begin{enumerate}[(a)]
  \item  $U(\eta)<p-1$,  $p\leq 2$,  and
         $U(\eta)-\frac{L(\eta)}{p-1}<p-1+\frac{sp}{p-1}$;
  \item  $U(\eta)<p-1$, $p\geq 2 $,  and $U(\eta)-L(\eta)<ps+1$;
  \item  $p-1\leq U(\eta)<ps+p$, \ $s>0$, and $L(\eta)>p-1-ps$.
\end{enumerate}
Then the following conditions are equivalent:
\begin{enumerate}[(i)]
\item  $h_{s,f}: B_p(\eta_{[2-p]}) \to L^p(\D, \eta dA)$ is a  bounded operator;
\item $h_F^{s}: B_p(\eta_{[2-p]}) \to S_p(\eta_{[2-p]}) $  is a  bounded operator, where $F$ is in $H(\D)$ satisfying  $F'=f$ on $\D$;
\item $d\mu(z)=|f(z)|^p\eta(z) dA(z)$ is a $p$-Carleson measure for $B_p(\eta_{[2-p]})$.
\end{enumerate}
\end{theor}

\begin{proof}
$(i)\Leftrightarrow (ii)$.  For $g\in \mathcal P$, it is easy to check  that for any $z\in \D$,
$$
\frac{\partial h_F^{s}g}{\partial z}(z)=0
$$ and
\begin{align*}
\frac{\partial h_F^{s}g}{\partial \overline{z}}(z)&=(s+1)\overline{\ind \overline{g(w)}f(w)\frac{(1-|w|^2)^s}{(1-\overline{w}z)^{s+2}} dA(w)}\\
&=h_{s,f} g(z).
\end{align*}
Hence
$$
|\nabla  h_F^{s}g(z)|\thickapprox |h_{s,f} g(z)|
$$
for all $z\in \D$. Then the equivalence between (i) and (ii) follows.

$(iii)\Rightarrow (i)$.  Suppose $\mu$ is a $p$-Carleson measure for $B_p(\eta_{[2-p]})$. Since $\eta\in \B_{p,s}$, the Bergman projection $P_s: L^p(\D, \eta dA) \to A_\eta^p$ is bounded. Then, for any $g\in B_p(\eta_{[2-p]})$, we get
\begin{align*}
\|h_{s,f}g\|_{L^p(\D, \eta dA)}^p
&=\|P_s(f\overline{g})\|_{L^p(\D, \eta dA)}^p
\lesssim \|fg\|_{L^p(\D, \eta dA)}^p
\lesssim \|g\|_{B_p(\eta_{[2-p]})}^p.
\end{align*}
Thus $h_{s,f}$ is bounded.

$(i)\Rightarrow (iii)$.
Suppose $h_{s,f}:B_p(\eta_{[2-p]}) \to L^p(\D, \eta dA)$ is bounded.
Note that  any radial weight in $\B_{p,s}$  belongs to $\d$.  By (\ref{19-2}), Proposition \ref{w yiwan} and the remark after it, we have
$$\tilde{\eta}\in\mathcal{R}, \,\,\,L(\tilde{\eta})= L(\eta), \,\,\, U(\tilde{\eta})=U(\eta),$$
and
$$\|h_{s,f}\|_{B_p(\tilde{\eta}_{[2-p]}) \to L^p(\D, \tilde{\eta} dA) }\thickapprox \|h_{s,f}\|_{B_p(\eta_{[2-p]}) \to  L^p(\D, \eta dA)}. $$
Thus $h_{s,f}:\ B_p(\tilde{\eta}_{[2-p]})\to L^p(\D, \tilde{\eta} dA)$ is also a bounded operator.
Here,
$$\tilde{\eta}_{[2-p]}(z)=(1-|z|^2)^{2-p}\tilde{\eta}(z)=(1-|z|^2)^{2-p}\frac{\hat{\eta}(z)}{1-|z|},\,\,\,z\in\D.$$
We claim that
$$
\|f\overline{g}-\overline{h_{s,f}g}\|_{L^p(\D, \tilde{\eta} dA)}\lesssim \|g\|_{B_p(\tilde{\eta}_{[2-p]})}
$$
for all $g\in B_p(\tilde{\eta}_{[2-p]})$. By this claim, the boundedness of $h_{s,f}: B_p(\tilde{\eta}_{[2-p]}) \to L^p(\D, \tilde{\eta} dA)$, and Proposition \ref{w yiwan}, we see that
 $\mu$ is a $p$-Carleson measure for  $B_p(\eta_{[2-p]})$. Thus,  it suffices to prove the claim for case (a), case  (b) and case (c) respectively.
As explained in Lemma \ref{1021-1}, $P_sf=f$. This, together with (i), yields
\begin{align}\label{1021-2}
f(z)\overline{g(z)}-\overline{h_{s,f}g(z)}=(s+1)\int_\D \frac{f(w)\overline{(g(z)-g(w))}}{(1-\overline{w}z)^{2+s}}dA_s(w).
\end{align}
for any $z\in \D$.

Case (a).   Suppose
\begin{align}\label{i}
U(\eta)<p-1,\,\,\,\,  1<p\leq 2,\,\,\, U(\eta)-\frac{L(\eta)}{p-1}<p-1+\frac{sp}{p-1}.
\end{align}
For     $1<p<2$,    by  $L(\eta)\leq U(\eta)$ and $U(\eta)-\frac{L(\eta)}{p-1}<p-1+\frac{sp}{p-1}$, we get   $U(\eta)>\frac{(p-1)^2+sp}{p-2}$.
Note that  $U(\eta)<p-1$.  Then   $\frac{(p-1)^2+sp}{p-2}<p-1$; that is, $\frac{ps}{p-1}>-1$. Hence
$U(\eta)<p-1<p+\frac{ps}{p-1}$.
For  $p=2$, condition  (\ref{i}) also yields
$$U(\eta)<1+2s+L(\eta)\leq 1+2s+U(\eta)<2+2s=p+\frac{ps}{p-1}.$$
Thus for $1<p\leq 2$,
$$
U(\eta)<p-1<p+\frac{ps}{p-1}.
$$
Due to $s>-1$, it is clear that $U(\eta)<p+s$.
By a simple calculation, $U(\eta)<p-1$ is equivalent to
\begin{equation}\label{07212}
(2+s)p<4+(p-2+s)p-(U(\eta)-1)(p-1)+p-2;
\end{equation}
$U(\eta)<p+s$ is equivalent to
\begin{equation}\label{07213}
(p-2+s)p-(U(\eta)-1)(p-1)>U(\eta)-p-1;
\end{equation}
$U(\eta)<p+\frac{ps}{p-1}$ is equivalent to
\begin{equation}\label{07214}
(p-2+s)p-(U(\eta)-1)(p-1)>-1;
\end{equation}
$U(\eta)-\frac{L(\eta)}{p-1}<p-1+\frac{sp}{p-1}$ is equivalent to
\begin{equation}\label{07215}
(p-2+s)p-(U(\eta)-1)(p-1)>p-2-L(\eta).
\end{equation}
Note that $p>1$. By (\ref{07212}), (\ref{07213}), (\ref{07214}), (\ref{07215}) and $L(\eta)>0$,  there exits $\varepsilon>0$ such that
\begin{align} \label{07216}
(2+s)p\leq  4+(p-2+s)p-(U(\eta)-1+\varepsilon)(p-1)+p-2,
\end{align}
\begin{align}\label{07217}
(p-2+s)p-(U(\eta)-1+\varepsilon)(p-1)>\max\{U(\eta)-p-1,-1\},
\end{align}
and
\begin{align}\label{07218}
\min\left\{(p-2+s)p-(U(\eta)-1+\varepsilon)(p-1), p-2\right\}>p-2-L(\eta).
\end{align}
$U(\eta)>0$ and (\ref{07214}) yield
\begin{align}\label{07219}
p-2+s>\frac{(U(\eta)-1)(p-1)-1}{p}>-1
\end{align}
Because of  $\tilde{\eta}\in\mathcal{R}$, for any   positive number  $\varepsilon_1$,
$\frac{\tilde{\eta}(t)}{(1-t)^{U(\eta)-1+\varepsilon_1}}$ is essentially increasing on $[0,1)$. Choose $\varepsilon_1\leq \varepsilon$.   Combining this with  (\ref{07217}), we obtain
\begin{align}\label{072110}
&\ind \frac{(1-|w|^2)^{(p-2+s)p}}{\tilde{\eta}(w)^{p-1}}dA(w) \nonumber \\
\lesssim & \ind  (1-|w|^2)^{(p-2+s)p-(U(\eta)-1+\varepsilon_1)(p-1)} dA(w)
<\infty.
\end{align}
It is known  that $f=\overline{h_{s,f}1}\in A_{\eta}^p$. In fact, $f$ is also in $A_{\tilde{\eta}}^p$.   By  Lemma \ref{1021-1},
\begin{align}\label{19-3}
M_f:=\sup_{a\in\D}(1-|a|^2)|f(a)|<\infty.
\end{align}
By   $1<p\leq 2$,   (\ref{07219}), (\ref{072110}), (\ref{19-3}), the H\"older inequality and the essentially increasing property of  $\frac{\tilde{\eta}(t)}{(1-t)^{U(\eta)-1+\varepsilon}}$ on $[0,1)$, we deduce
{\small
\begin{align}
&\left|\int_\D \frac{f(w)\overline{(g(z)-g(w))}}{(1-\overline{w}z)^{2+s}}dA_s(w)\right|^p \nonumber \\
\leq& M_f^{(2-p)p} \left(\int_\D \frac{|f(w)|^{p-1}|g(z)-g(w)|}{|1-\overline{w}z|^{2+s}}dA_{p-2+s}(w)\right)^p \nonumber  \\
\leq& M_f^{(2-p)p} \|f\|_{A_{\tilde{\eta}}^p}^{p(p-1)} \int_\D  \frac{|g(z)-g(w)|^p(1-|w|^2)^{(p-2+s)p}}{|1-\overline{w}z|^{(2+s)p}{\tilde{\eta}}(w)^{p-1}}  dA(w) \nonumber \\
\lesssim&  M_f^{(2-p)p} \|f\|_{A_{\tilde{\eta}}^p}^{p(p-1)}  \int_\D \frac{|g(z)-g(w)|^p(1-|w|^2)^{(p-2+s)p-(U(\eta)-1+\varepsilon)(p-1)}}{|1-\overline{w}z|^{(2+s)p}}  dA(w).  \label{07203}
\end{align}
}
Then, it follows from (\ref{1021-2}), (\ref{07216}) and  (\ref{07203}) that
\begin{align}
& \|f\overline{g}-\overline{h_{s,f}g}\|_{L^p(\D, \tilde{\eta} dA)}^p \nonumber \\
\lesssim & \int_\D\int_\D \frac{|g(z)-g(w)|^p(1-|w|^2)^{(p-2+s)p-(U(\eta)-1+\varepsilon)(p-1)}}{|1-\overline{w}z|^{(2+s)p}}  {\tilde{\eta}}(z)dA(w) dA(z) \nonumber\\
\lesssim & \int_\D\int_\D \frac{|g(z)-g(w)|^p(1-|w|^2)^{(p-2+s)p-(U(\eta)-1+\varepsilon)(p-1)}}{|1-\overline{w}z|^{4+(p-2+s)p-(U(\eta)-1+\varepsilon)(p-1)+p-2}}  {\tilde{\eta}}(z)dA(w) dA(z). \label{072111}
\end{align}
 Bear  in mind (\ref{07217}) and (\ref{07218}).
Applying  Proposition \ref{R sigma tau} by setting  $\o(z)=\tilde{\eta}(z)(1-|z|^2)^{2-p}$, $\sigma=p-2$ and $\tau=(p-2+s)p-(U(\eta)-1+\varepsilon)(p-1)$, we see that
\begin{align*}
&\int_\D\int_\D \frac{|g(z)-g(w)|^p(1-|w|^2)^{(p-2+s)p-(U(\eta)-1+\varepsilon)(p-1)}}{|1-\overline{w}z|^{4+(p-2+s)p-(U(\eta)-1+\varepsilon)(p-1)+p-2}}  {\tilde{\eta}}(z)dA(w) dA(z)\\
\lesssim & \|g\|_{B_p({\tilde{\eta}}_{[2-p]})}^p.
\end{align*}
This, together with (\ref{072111}), yields that our claim holds.

 Case (b). The result for  $p=2$ has been proved in  case (a). Now  suppose
\begin{align}\label{ii}
U(\eta)<p-1,\,\,\,  p> 2,\,\,\,s>-1, \ \  U(\eta)-L(\eta)<ps+1.
\end{align}
 Note that $0<L(\eta)\leq U(\eta)<\infty$.
Then $ps> U(\eta)-L(\eta)-1\geq -1$, and hence   $U(\eta)<p-1<p+ps$. Also, $2U(\eta)<2p-2<2p+ps$.
By these inequalities and  (\ref{ii}), there exist real numbers $x$, $y$, and positive number $\varepsilon$  such that
\begin{align}\label{07221}
U(\eta)\leq p-1-\varepsilon,
\end{align}
\begin{align}\label{07222}
&\max\{U(\eta)-p+1+\varepsilon, \ \ 1-p\} \nonumber \\
<&  (p'+x)(p-1) \nonumber \\
<& \min\left\{ps+p+1-U(\eta),ps+1,ps-p+2+L(\eta)\right\},
\end{align}
and
\begin{align}\label{07223}
 2p-3-U(\eta)-\varepsilon<(y-(p'+x))(p-1)<2p-3-U(\eta)+L(\eta)-\varepsilon.
 \end{align}
 From  (\ref{07222}), $p'+x>-1$ and $ps-(p-1)(p'+x)>-1$. Using   H\"older's inequality, we get
\begin{align}\label{07224}
&\left|\int_\D \frac{f(w)\overline{(g(z)-g(w))}}{(1-\overline{w}z)^{2+s}}dA_s(w)\right|^p \nonumber \\
\leq&\left(\int_\D \frac{|f(w)|^{p'}}{|1-\overline{w}z|^y} dA_{p'+x}(w)\right)^{p-1} \int_\D \frac{|g(z)-g(w)|^pdA_{ps-(p-1)(p'+x)}(w)}{|1-\overline{w}z|^{(2+s)p-(p-1)y}}.
\end{align}
Note that $\frac{1}{\f{p-1}{p-2}}+\f{1}{p-1}=1$, $\f{p-1}{p-2}>1$ and $p-1>1$. It follows from the H\"older inequality again that
\begin{align}\label{07225}
&\left(\int_\D \frac{|f(w)|^{p'}}{|1-\overline{w}z|^y} dA_{p'+x}(w)\right)^{p-1} \nonumber \\
\leq& \|f\|_{A_{\tilde{\eta}}^p}^p
\left(\int_\D \left(\frac{(1-|w|^2)^{p'+x}}{|1-\ol{w}z|^y}\right)^\frac{p-1}{p-2}{\tilde{\eta}}(w)^{-\frac{1}{p-2}}dA(w)\right)^{p-2}.
\end{align}
Because of   $\tilde{\eta}\in\mathcal{R}$,   $\frac{\tilde{\eta}(t)}{(1-t)^{U(\eta)-1+\varepsilon}}$ is essentially increasing on $[0,1)$. Hence
\begin{align}\label{07226}
\sup\limits_{w\in\D}\frac{(1-|w|)^{U(\eta)-1+\varepsilon}}{\tilde{\eta}(w)}<\infty.
\end{align}
Due to  (\ref{07222}) and (\ref{07223}),  we get
\begin{align}\label{1104-2}
\frac{(p'+x) (p-1)-(U(\eta)-1+\varepsilon)}{p-2}>-1,
\end{align}
and
\begin{align}\label{1104-3}
\frac{y(p-1)}{p-2}-\frac{(p'+x) (p-1)-(U(\eta)-1+\varepsilon)}{p-2}>2,
\end{align}
By (\ref{07225}), (\ref{07226}), (\ref{1104-2}), (\ref{1104-3}) and Lemma \ref{well-known estimate}, one gets
\begin{align*}
&\left(\int_\D \frac{|f(w)|^{p'}}{|1-\overline{w}z|^y} dA_{p'+x}(w)\right)^{p-1}\\
\lesssim&\|f\|_{A_{\tilde{\eta}}^p}^p
\left(\int_\D \frac{(1-|w|^2)^{\frac{(p'+x)(p-1)-(U(\eta)-1+\varepsilon)}{p-2}}}{|1-\ol{w}z|^\frac{y(p-1)}{p-2}}dA(w)\right)^{p-2}\\
\lesssim& \frac{\|f\|_{A_{\tilde{\eta}}^p}^p }{(1-|z|)^{(y-(p'+x))(p-1)+(U(\eta)-1+\varepsilon)-2(p-2)}} .
\end{align*}
Combining this with (\ref{07224}), we obtain
\begin{align} \label{07227}
&\int_\D \left|\int_\D \frac{f(w)\overline{(g(z)-g(w))}}{(1-\overline{w}z)^{2+s}}dA_s(w)\right|^p \tilde{\eta}(z)dA(z) \nonumber \\
\lesssim & \|f\|_{A_{\tilde{\eta}}^p}^p \int_\D
\int_\D \frac{|g(z)-g(w)|^p}{|1-\overline{w}z|^{(2+s)p-(p-1)y}}dA_{ps-(p-1)(p'+x)}(w) \nonumber \\
&\times \tilde{\eta}(z)dA_{2(p-2)-(y-(p'+x))(p-1)-(U(\eta)-1+\varepsilon)}(z).
\end{align}
Set   $\o(z)=\tilde{\eta}(z)(1-|z|^2)^{2-p}$,  $\sigma =3(p-2)-(y-(p'+x))(p-1)-(U(\eta)-1+\varepsilon)$,  and $\tau =ps-(p-1)(p'+x)$. Note that (\ref{07221}) holds if and only if
$$(2+s)p-(p-1)y \leq 4+\sigma+\tau. $$
Then (\ref{07227}) gives
\begin{align} \label{07228}
&\int_\D \left|\int_\D \frac{f(w)\overline{(g(z)-g(w))}}{(1-\overline{w}z)^{2+s}}dA_s(w)\right|^p \tilde{\eta}(z)dA(z) \nonumber \\
\lesssim & \|f\|_{A_{\tilde{\eta}}^p}^p \int_\D \int_\D \frac{|g(z)-g(w)|^p}{|1-\overline{w}z|^{4+\sigma+\tau}}\o(z) dA_\tau(w)dA_\sigma(z).
\end{align}
By   (\ref{07222}) and (\ref{07223}), condition (\ref{07231}) in Proposition \ref{R sigma tau} holds. By (\ref{07223}),  $L(\eta)=L(\tilde{\eta})$,  and $\tilde{\eta}\in\mathcal{R}$, we get
\begin{align*}
\ind \o_{[\sigma]}(z) dA(z)&=\ind \tilde{\eta}(z)(1-|z|^2)^{2(p-2)-(y-(p'+x))(p-1)-(U(\eta)-1+\varepsilon)}   dA(z)  \\
&\lesssim \ind \tilde{\eta}(z)(1-|z|^2)^{-L(\eta)}  dA(z)<\infty,
\end{align*}
which means that $\o_{[\sigma]}$ is a weight. Thus,  Proposition \ref{R sigma tau} yields
\begin{align} \label{07229}
\int_\D \int_\D \frac{|g(z)-g(w)|^p}{|1-\overline{w}z|^{4+\sigma+\tau}}\o(z) dA_\tau(w)dA_\sigma(z)
\lesssim \|g\|_{B_p({\tilde{\eta}}_{[2-p]})}^p.
\end{align}
From (\ref{1021-2}), (\ref{07228}) and (\ref{07229}),  our claim holds.

 Case (c).   Note that
\begin{align*}
  U(\eta)<ps+p,\,\,\,p>1,\,\,\,s>0,\,\,\, L(\eta)>p-1-ps.
\end{align*}
Then there exits a real number $x$ satisfying
\begin{equation}\label{0724-1}
-1<x<\min\left\{\frac{ps-U(\eta)+1}{p-1}, \frac{ps-p+1}{p-1}, \f{ps-2p+2+L(\eta)}{p-1}\right\}.
\end{equation}
Clearly, there is another real number $y$ such that
\begin{equation}\label{0724-2}
0<y-x-2<\min  \left\{\frac{p}{p-1},  \f{L(\eta)}{p-1}\right\}.
\end{equation}
Note that  $p'+x>x>-1$ and $ps-p-(p-1)x>-1$.
 From   H\"older's inequality, (\ref{19-3}), (\ref{0724-2}),  and Lemma \ref{well-known estimate},  one gets
\begin{align}\label{0724-3}
&\left|\int_\D \frac{f(w)\overline{(g(z)-g(w))}}{(1-\overline{w}z)^{2+s}}dA_s(w)\right|^p \nonumber \\
\leq&\left(\int_\D \frac{|f(w)|^{p'}}{|1-\overline{w}z|^y} dA_{{p'}+x}(w)\right)^{p-1}
\int_\D \frac{|g(z)-g(w)|^p}{|1-\overline{w}z|^{(2+s)p-(p-1)y}}dA_{ps-p-(p-1)x}(w) \nonumber \\
\lesssim& M_f^p  \left(\int_\D \frac{dA_{x}(w)}{|1-\overline{w}z|^y} \right)^{p-1}
\int_\D \frac{|g(z)-g(w)|^p}{|1-\overline{w}z|^{(2+s)p-(p-1)y}}dA_{ps-p-(p-1)x}(w) \nonumber \\
\lesssim& M_f^p  (1-|z|)^{-(y-x-2)(p-1)} \int_\D \frac{|g(z)-g(w)|^p}{|1-\overline{w}z|^{(2+s)p-(p-1)y}}dA_{ps-p-(p-1)x}(w).
\end{align}
Set  $\tau=ps-p-(p-1)x$, $\sigma=p-2-(y-x-2)(p-1)$ and $\o(z)=\tilde{\eta}(z)(1-|z|^2)^{2-p}$. Then
$(2+s)p-(p-1)y=4+\tau+\sigma$. Also, (\ref{0724-1}) and (\ref{0724-2}) yield condition (\ref{07231}) in Proposition \ref{R sigma tau} holds. Note that $0<y-x-2<\f{L(\eta)}{p-1}$ and $\tilde{\eta}\in\mathcal{R}$. Then
\begin{align*}
\ind \o_{[\sigma]}(z) dA(z)&=\ind \tilde{\eta}(z)(1-|z|^2)^{-(y-x-2)(p-1)}   dA(z)  \\
&\lesssim \ind \tilde{\eta}(z)(1-|z|^2)^{-L(\eta)}  dA(z)<\infty;
\end{align*}
that is, $\o_{[\sigma]}$ is a weight.  By (\ref{0724-3}) and applying  Proposition \ref{R sigma tau}, we have
\begin{align*}
&\int_\D    \left|\int_\D \frac{f(w)\overline{(g(z)-g(w))}}{(1-\overline{w}z)^{2+s}}dA_s(w)\right|^p  {\tilde{\eta}}(z)dA(z)\\
\lesssim & M_f^p \int_\D \int_\D \frac{|g(z)-g(w)|^p}{|1-\overline{w}z|^{(2+s)p-(p-1)y}}dA_{ps-p-(p-1)x}(w)  {\tilde{\eta}}(z)dA_{-(y-x-2)(p-1)}(z)\\
\thickapprox &   M_f^p  \int_\D \int_\D \frac{|g(z)-g(w)|^p}{|1-\overline{w}z|^{4+\tau+\sigma}}\o(z) dA_{\tau}(w)  dA_{\sigma}(z)\\
\lesssim& M_f^p \int_\D |g^\prime(z)|^p{\tilde{\eta}}(z)dA(z).
\end{align*}
Combining this with (\ref{1021-2}), we see that  our claim also holds for this case.    The proof is complete.
\end{proof}

\noindent{\bf  Remark.  }  First,   the proof of case (c) is valid for $U(\eta)<ps+p,\,\,\,p>1,\,\,\,s>0,\,\,\, L(\eta)>p-1-ps$.  It is easy to see that the case of $U(\eta)<p-1,\,\,\,p>1,\,\,\,s>0,\,\,\, L(\eta)>p-1-ps$ is  contained in case (a)  and case (b). Thus we use the now  version of case (c)   in the expressions of  Theorem  \ref{3main}.  In addition, by the proof  of  Theorem  \ref{3main},  $h_{s,f}: B_p(\eta_{[2-p]}) \to L^p(\D, \eta dA)$ is a  bounded operator if and only if $|f(z)|^p\tilde{\eta}(z) dA(z)$ is a $p$-Carleson measure for $B_p(\tilde{\eta}_{[2-p]})$.
When $p>1$ and $\o$ is a B\'ekoll\'e-Bonami weight  satisfying  one more condition including  the case of   $\o$ in $\mathcal R$,  the  description of $p$-Carleson measures for  $B_p(\o)$  can be found  in \cite{ARS}.

Checking the proof of  case (a) and case (b) in  (i)$\Rightarrow$(iii) of Theorem  \ref{3main},
if $\eta$ satisfies that $\frac{\hat{\eta}(t)}{(1-t)^{U(\eta)}}$ is essentially  increasing  on [0, 1), we can choose $\varepsilon=0$ in the proof.
For such  case,  the condition  $U(\eta)<p-1$ in both case (a) and case (b) in Theorem  \ref{3main} can be replaced by $U(\eta)\leq p-1$.  We state this result as follows.

\begin{prop}\label{0728-1}
Let  $1<p<\infty$, $-1<s<\infty$, $f\in H(\D)$ and  let $\eta\in \B_{p, s}$ be a radial weight such that $\frac{\hat{\eta}(t)}{(1-t)^{U(\eta)}}$ is essentially  increasing  on [0, 1).   Suppose  $\eta,p,s$ satisfies   further one of the following conditions:
\begin{enumerate}[(a)]
  \item  $U(\eta)\leq p-1$,  $p\leq 2$,  and
         $U(\eta)-\frac{L(\eta)}{p-1}<p-1+\frac{sp}{p-1}$;
  \item  $U(\eta)\leq p-1$, $p\geq 2 $,  and $U(\eta)-L(\eta)<ps+1$.
\end{enumerate}
Then the following conditions are equivalent:
\begin{enumerate}[(i)]
\item  $h_{s,f}: B_p(\eta_{[2-p]}) \to L^p(\D, \eta dA)$ is a  bounded operator;
\item $h_F^{s}: B_p(\eta_{[2-p]}) \to S_p(\eta_{[2-p]}) $  is a  bounded operator, where $F$ is in $H(\D)$ satisfying  $F'=f$ on $\D$;
\item $d\mu(z)=|f(z)|^p\eta(z) dA(z)$ is a $p$-Carleson measure for $B_p(\eta_{[2-p]})$.
\end{enumerate}
\end{prop}

The following result is  Theorem 4.2 in \cite{BP}.

\begin{otherth}\label{BPtheorem}
Let $1<p<\infty$ and $\alpha\leq 1/2$, and let $s$ with $s>\max\{\frac{-1}{p}, \frac{1-p}{p}\}$ if $\alpha=1/2$ and $s>\max\{0, -2\alpha\}$ if $\alpha<1/2$. Let $g$ be analytic on $\D$. Then the operator $h_{s, g}$ is bounded from $B_p(\alpha)$ to $L^p(\D, (1-|z|^2)^{p-1-2\alpha}dA(z))$ if and only if the measure
$|g(z)|^p(1-|z|^2)^{p-1-2\alpha}dA(z)$ is a $p$-Carleson measure for $B_p(\alpha)$.
\end{otherth}

Applying Theorem \ref{3main} and Proposition \ref{0728-1}, we obtain Theorem \ref{4main} below   which optimizes the range of the   parameter $\alpha$ in  Theorem \ref{BPtheorem}.

\begin{theor}\label{4main}
Suppose $g\in H(\D)$,  $1<p<\infty$, $\alpha<\frac{p}{2}$, and  $s>-\frac{2\alpha}{p}$. Furthermore, if $p$, $\alpha$ and $s$ satisfy further  one of the following conditions:
\begin{enumerate}[(a)]
  \item $\alpha\geq \frac{1}{2}$,  $p\leq 2$  and $s>\frac{(4-2p)\alpha-1}{p}$;
  \item  $\alpha\geq \frac{1}{2}$, $p>2$ and $s>-\frac{1}{p}$;
  \item  $\alpha \leq \frac{1}{2}$  and $s>0$.
\end{enumerate}
Then $h_{s, g}$ is bounded from $B_p(\alpha)$ to $L^p(\D, (1-|z|^2)^{p-1-2\alpha}dA(z))$  if and only if the measure
$|g(z)|^p(1-|z|^2)^{p-1-2\alpha}dA(z)$ is a $p$-Carleson measure for $B_p(\alpha)$.
\end{theor}
\begin{proof}
Set $\eta(z)=(1-|z|^2)^{p-1-2\alpha}$ with $p>2\alpha$. Then $B_p(\eta_{[2-p]})=B_p(\alpha)$.
Clearly, $U(\eta)=L(\eta)=p-2\alpha$. For $-1<s<\infty$ and $a\in \D\setminus \{0\}$, we deduce
\begin{align*}
&\left(\int_{S(a)}\eta(z)  dA(z) \right) \left(\int_{S(a)} \left(\frac{\eta(z)}{(1-|z|^2)^s}\right) ^{-\frac{p'}{p}} dA_s(z) \right)^{\frac{p}{p'}}\\
\thickapprox& (1-|a|)^{p+1-2\alpha} \left( (1-|a|) \int_{|a|}^1 (1-r)^{-1+s+\frac{2\alpha+s}{p-1}} dr  \right)^{\frac{p}{p'}}\\
\thickapprox& (1-|a|)^{(2+s)p}\thickapprox (A_s (S(a)))^p,
\end{align*}
where we use $s+\frac{2\alpha+s}{p-1}=\frac{2\alpha+ps}{p-1}>0$. Hence $\eta\in \B_{p,s}$ when $s>\frac{-2\alpha}{p}$. Note that $\frac{-2\alpha}{p}>-1$.  Then the desired result follows from Theorem \ref{3main}  and Proposition \ref{0728-1}.
\end{proof}

\noindent{\bf  Remark.  }
 Theorem \ref{4main} contains  the case of $1<p<\infty$ and
$1/2<\alpha <p/2$ missing in Theorem \ref{BPtheorem}. For  $1<p<\infty$ and $\alpha<1/2$, by (c) of Theorem \ref{4main},  we get the  same  result in  Theorem \ref{BPtheorem}. For $1<p<\infty$ and $\alpha=1/2$, from (b) and (c) in  Theorem \ref{4main}, we also get  the same result as Theorem \ref{BPtheorem}.

 Z. Lou and  R. Qian \cite{LQ} investigated  the boundedness of Hankel type operator $h_{s,f}$ related to a class of Dirichlet type spaces $\mathfrak{D}_\rho$. Let $\rho: [0, \infty)\rightarrow [0, \infty)$ be a right-continuous and nondecreasing function. $\mathfrak{D}_\rho$ is equal to $B_2(\o_\rho)$ with $\o_\rho(|z|)=\rho(1-|z|)$. $\rho$ is said to be upper (resp. lower) type $\gamma\in (0, \infty)$ (cf. \cite{J}) if
 $$
 \rho(xy)\leq C x^\gamma \rho(y), \ \ x\geq 1 \ \ (\text{resp.}\ \  x\leq 1) \  \ \text{and}\ \ 0<y<\infty.
 $$
 Clearly, if $\rho$ is  upper type $\gamma$ for some $\gamma>0$, then $\rho(2y)\lesssim \rho(y)$ for all $y>0$.

The following result is Theorem 1 in \cite{LQ}.
\begin{otherth}\label{LQtheorem}
Let $0<\gamma<1$,  $s>\frac{1+\gamma}{2}$ and  $f\in H(\D)$. Suppose $\rho: [0, \infty)\rightarrow [0, \infty)$ is  a right-continuous and nondecreasing function of  upper type $\gamma$.
Then $h_{s, f}: B_2(\o_\rho)\to  L^2(\D, \o_\rho dA)$ is bounded if and only if $|f(z)|^2\rho(1-|z|^2)dA(z)$ is a $2$-Carleson measure for $B_2(\o_\rho)$.
\end{otherth}

In \cite[p. 219, Remark 2]{LQ}, the authors  mentioned that they failed to prove Theorem \ref{LQtheorem} and other results in their paper without the condition $s>\frac{1+\gamma}{2}$.
Motivated by this remark, we  apply Theorem \ref{3main} to give  the following result which means that the condition  in Theorem \ref{LQtheorem}   can be improved.

\begin{theor}\label{5main}
Let $0<\gamma<\infty$,  $s>\max\{0, \frac{\gamma-1}{2}\}$ and  $f\in H(\D)$. Suppose $\rho: [0, \infty)\rightarrow [0, \infty)$ is  a right-continuous and nondecreasing function of upper type $\gamma$.
Then $h_{s, f}: B_2(\o_\rho)\to L^2(\D, \o_\rho dA) $ is bounded if and only if $|f(z)|^2\rho(1-|z|^2)dA(z)$ is a $2$-Carleson measure for $B_2(\o_\rho)$.
\end{theor}
\begin{proof}
Since $\rho$ is a nondecreasing function, we obtain
$$
\widehat{\o_\rho}(r)=\int_r^1 \rho(1-s)ds\leq (1-r)\rho(1-r),
$$
and
$$
\widehat{\o_\rho}(r)\geq \int_r^{\frac{1+r}{2}} \rho(1-s)ds
\geq \frac{1-r}{2}\rho\left(\frac{1-r}{2}\right)
\gtrsim  (1-r)\rho(1-r)
$$
for all $r\in [0, 1)$.  Hence $\widehat{\o_\rho}(r)\thickapprox (1-r)\rho(1-r)$ for all  $r\in [0, 1)$.  Since $\rho$ is upper type $\gamma$, $\rho(1-x)\lesssim (\frac{1-x}{1-y})^\gamma \rho(1-y)$ for all $0<x\leq y<1$.
Consequently,  $\frac{\widehat{\o_\rho}(r)}{(1-r)^{\gamma+1}}$ is  essentially increasing on $[0, 1)$ and $\frac{\widehat{\o_\rho}(r)}{1-r}$ is  essentially decreasing on $[0, 1)$. Thus $1\leq L(\o_\rho)\leq U(\o_\rho)\leq \gamma+1$.

Since $s>\frac{\gamma-1}{2}$ and $\rho(1-x)/(1-x)^\gamma$ is essentially increasing on $[0, 1)$,   we deduce
\begin{align*}
&\int_{S(a)}\rho(1-|z|)  dA(z)  \int_{S(a)} \left(\frac{\rho(1-|z|)}{(1-|z|^2)^s}\right) ^{-1} dA_s(z) \\
\thickapprox& (1-|a|)^2\widehat{\o_\rho}(|a|) \int_{|a|}^1 \frac{(1-r)^{2s+\gamma}}{\rho(1-r)(1-r)^\gamma}dr\\
\lesssim& (1-|a|)^{3+\gamma} \int_{|a|}^1 (1-r)^{2s-\gamma}dr
\thickapprox (1-|a|)^{4+2s}\thickapprox (A_s(S(a)))^2
\end{align*}
for all $a\in \D\setminus \{0\}$. This gives $\o_\rho \in \B_{2, s}$ when $s>\frac{\gamma-1}{2}$.

Set $p=2$ and  $\eta=\o_\rho$ in (c) of  Theorem \ref{3main}. Since $s>\max\{0, \frac{\gamma-1}{2}\}$ and
$1\leq L(\eta)\leq U(\eta)\leq \gamma+1$, we see that (c) in   Theorem \ref{3main} holds. Then we get the desired result.
\end{proof}

\subsection{The compactness of Hankel type operators  from  $B_p(\o)$  to $S_p(\o)$}  This subsection is devoted to consider the compactness of  Hankel type operators    corresponding to Theorem \ref{3main}. This result is also  new for   $\o(z)=(1-|z|^2)^{1-2\alpha}$.

A linear operator from one normed linear space  to another normed linear space is compact if the image of the unit ball under the operator has compact closure.

We begin with  the following lemma.
\begin{limma} \label{0729-1}
Let  $1\leq p<\infty$,  $-1<s<\infty$ and  $f \in H(\D)$.  Suppose $\eta$ is a weight and  there
exist two constants $r\in(0,1)$ and $C>0$ such that
$$
C^{-1}\eta (\zeta)\le\eta (z)\le C\eta (\zeta)
$$
for all $z$ and $\zeta$ satisfying $\rho(z,\zeta)<r$.   Then the following statements hold:
\begin{enumerate}[(i)]
  \item  $h_{s,f}: B_p(\eta_{[2-p]}) \to L^p(\D, \eta dA)$ is a  compact  operator if and only if\\ $\|h_{s,f} g_n\|_{L^p(\D, \eta dA)}\rightarrow 0$ as $n\to \infty$ whenever $\{g_n\}$ is a bounded sequence in $B_p(\eta_{[2-p]})$ that
     converges to 0 uniformly on every compact subset of $\D$;
  \item $I_d: B_p(\eta_{[2-p]}) \to  L^p(\D, |f(z)|^p\eta(z) dA(z)) $ is a  compact  operator  if and only if   $\|g_n\|_{L^p(\D, |f(z)|^p\eta(z) dA(z))}\rightarrow 0$ as $n\to \infty$  whenever $\{g_n\}$ is a bounded sequence in $B_p(\eta_{[2-p]})$ that converges to 0 uniformly on every compact subset of $\D$.
  \end{enumerate}
\end{limma}
\begin{proof}
 For any $g\in H(\D)$, the assumes on $\eta$ and the subharmonicity of $|g|^p$ yield
\begin{align}\label{32-1}
  \ind |g'(\zeta)|^p  \eta(\zeta) dA(\zeta) &\geq  \int_{\Delta(z, r)} |g'(\zeta)|^p  \eta(\zeta) dA(\zeta) \nonumber \\
  &\geq C^{-1} \eta(z) \int_{\Delta(z, r)} |g'(\zeta)|^p   dA(\zeta) \nonumber \\
  &\geq C_1 |g'(z)|^p  \eta(z) (1-|z|)^2
\end{align}
for all $z\in \D$, where $C_1$ is a positive constant independent of $z$ and $g$. Clearly, $g(z)=\int_0^z g'(\zeta)d\zeta+g(0)$.  Hence  the point evaluations on $B_p(\eta_{[2-p]})$ are bounded  functionals.  Similarly,
the point evaluations on $A^p_\eta$ are also bounded.  For $z\in \D$ satisfying  $f(z)\not =0$,  the point evaluation induced by $z$ on   $A^p_{|f|^p\eta}$ is also bounded. Next, the proof of this lemma   is quite standard (cf. \cite[Proposition 3.11]{CM}). We omit it.
\end{proof}

We also need the following lemma.

\begin{limma} \label{0729-2}
Let  $1<p<\infty$.   Suppose $\eta \in \check{\d}$  and  there
exist two constants $r\in(0,1)$ and $C>0$ such that
$$
C^{-1}\eta (\zeta)\le\eta (z)\le C\eta (\zeta)
$$
for all $z$ and $\zeta$ satisfying $\rho(z,\zeta)<r$.    Then $I_d: B_p(\eta_{[2-p]}) \to  L^p(\D, \eta(z) dA(z)) $ is a  compact  operator.
\end{limma}
\begin{proof}
Let $\{g_n\}$ be a bounded sequence in $B_p(\eta_{[2-p]})$ that converges to 0 uniformly on every compact subset of $\D$. Similar to  (\ref{32-1}), we get
\begin{align*}
\infty >  \sup_n \ind |g'_n(\zeta)|^p \eta(\zeta) dA(\zeta) \gtrsim \sup_{n} \sup_{z\in \D} |g_n'(z)|^p  \eta(z) (1-|z|^2)^2;
\end{align*}
that is,
\begin{align*}
 |g_n'(z)|^p  \eta(z) \lesssim G(z)
\end{align*}
for all $n$ and all $z\in \D$, where $G(z)=(1-|z|^2)^{-2}$.  For $p>1$, it is clear that the integral  $\ind G(\zeta) (1-|\zeta|^2)^p dA(\zeta)$ is convergent. Then Lebesgue's Dominated Convergence Theorem gives
\begin{align*}
\lim_{n \to \infty}\ind |g'_n(\zeta)|^p \eta(\zeta) (1-|\zeta|^2)^p  dA(\zeta)=0.
\end{align*}
Note that  $\eta \in \check{\d}$. By \cite[p. 8]{PR2021},
\begin{align*}
\ind |g_n(\zeta)|^p \eta(\zeta)  dA(\zeta) \lesssim |g_n(0)|^p + \ind |g'_n(\zeta)|^p \eta(\zeta) (1-|\zeta|^2)^p  dA(\zeta)
\end{align*}
for all $n$. Consequently,
\begin{align*}
\lim_{n \to \infty} \ind |g_n(\zeta)|^p \eta(\zeta)  dA(\zeta)=0.
\end{align*}
By Lemma \ref{0729-1}, the proof is complete.
\end{proof}

Now we give the main result in this subsection.

\begin{theor}\label{0729-3}
Suppose $1<p<\infty$, $-1<s<\infty$, $f\in H(\D)$ and $\eta\in \B_{p, s}$ is a radial weight. Let $\eta,p,s$ satisfy  further one of the following conditions:
\begin{enumerate}[(a)]
  \item  $U(\eta)<p-1$,  $p\leq 2$,  and
         $U(\eta)-\frac{L(\eta)}{p-1}<p-1+\frac{sp}{p-1}$;
  \item  $U(\eta)<p-1$, $p\geq 2 $,  and $U(\eta)-L(\eta)<ps+1$;
  \item  $p-1\leq U(\eta)<ps+p$, \ $s>0$, and $L(\eta)>p-1-ps$.
\end{enumerate}
Then the following conditions are equivalent:
\begin{enumerate}[(i)]
\item  $h_{s,f}: B_p(\eta_{[2-p]}) \to L^p(\D, \eta dA)$ is a  compact  operator;
\item $h_F^{s}: B_p(\eta_{[2-p]}) \to S_p(\eta_{[2-p]}) $  is a  compact  operator, where $F$ is in $H(\D)$ satisfying  $F'=f$ on $\D$;
\item $d\mu(z)=|f(z)|^p\eta(z) dA(z)$ is a vanishing  $p$-Carleson measure for $B_p(\eta_{[2-p]})$.
\end{enumerate}
\end{theor}
\begin{proof}
Similar to the proof of Theorem \ref{0729-3},  (i) and (ii) are equivalent.

Recall that a radial weight $\eta\in \B_{p, s}$   yields $\eta \in \mathcal D$.  Hence $\tilde{\eta} \in \mathcal R$. From \cite[p. 8]{PR2021},  $\tilde{\eta}(z) \thickapprox \tilde{\eta}(\zeta)$ whenever
$\rho(z, \zeta)<r$ with any fixed $r\in (0, 1)$.

$(iii)\Rightarrow (i)$. It follows from (iii) and Proposition  \ref{w yiwan} that $d\nu(z)=|f(z)|^p\tilde{\eta}(z) dA(z)$ is also a vanishing  $p$-Carleson measure for $B_p(\tilde{\eta}_{[2-p]})$.
Let $\{g_n\}$ be a bounded sequence in $B_p(\tilde{\eta}_{[2-p]})$ that converges to 0 uniformly on every compact subset of $\D$.  By Lemma \ref{0729-1},  $\|g_n\|_{L^p(\D, d\nu(z))}\rightarrow 0$ as $n\to \infty$.
  From the proof of ``$(iii)\Rightarrow (i)$" in Theorem \ref{0729-3} and Proposition  \ref{w yiwan} again, we know
  $$
\|h_{s,f} g_n\|^p_{L^p(\D, \tilde{\eta} dA)}\lesssim \|f g_n\|^p_{L^p(\D, \tilde{\eta} dA)}=\|g_n\|^p_{L^p(\D, d\nu)}.
$$
Then $\|h_{s,f} g_n\|^p_{L^p(\D, \tilde{\eta} dA)} \to 0$ as $n\to \infty$.  From  Lemma \ref{0729-1},  $h_{s,f}: B_p(\tilde{\eta}_{[2-p]}) \to L^p(\D, \tilde{\eta }dA)$ is a  compact  operator. Combining  this with
Proposition  \ref{w yiwan}, we get the desired result.

$(i)\Rightarrow (iii)$. Case (a).   Because of Proposition  \ref{w yiwan}, it suffices to show that $d\nu(z)=|f(z)|^p\tilde{\eta}(z) dA(z)$ is  a vanishing  $p$-Carleson measure for $B_p(\tilde{\eta}_{[2-p]})$.  Suppose  $\{g_n\}$ is  a bounded sequence in $B_p(\tilde{\eta}_{[2-p]})$ that converges to 0 uniformly on every compact subset of $\D$.  Note that $h_{s,f}: B_p(\tilde{\eta}_{[2-p]}) \to L^p(\D, \tilde{\eta }dA)$ is also a  compact  operator. Due to Lemma \ref{0729-1},
\begin{equation}\label{38-1}
\|h_{s,f} g_n\|^p_{L^p(\D, \tilde{\eta} dA)} \to 0 \ \ \text {as}\ \  n\to \infty.
\end{equation}
Note that $f\in A^p_{\tilde{\eta}}$. For any $\epsilon>0$, there exists $r$ in (0, 1) such that
\begin{equation}\label{38-2}
\left(\int_{\D \setminus r\D} |f(z)|^p \tilde{\eta}(z)dA(z)\right)^{p-1}<\epsilon.
\end{equation}
Checking Case (a) in  the proof of  ``$(i)\Rightarrow (iii)$" of  Theorem \ref{0729-3}, we get
{\small
\begin{align} \label{38-3}
&\left|\int_\D \frac{f(w)\overline{(g_n(z)-g_n(w))}}{(1-\overline{w}z)^{2+s}}dA_s(w)\right|^p \nonumber \\
\leq& M_f^{(2-p)p} \left(\left(\int_{r\D}+ \int_{\D \setminus r\D}\right) \frac{|f(w)|^{p-1}|g_n(z)-g_n(w)|}{|1-\overline{w}z|^{2+s}}dA_{p-2+s}(w)\right)^p \nonumber  \\
\lesssim & M_f^{(2-p)p} \|f\|_{A_{\tilde{\eta}}^p}^{p(p-1)} \int_{r\D}  \frac{|g_n(z)-g_n(w)|^p(1-|w|^2)^{(p-2+s)p}}{|1-\overline{w}z|^{(2+s)p}{\tilde{\eta}}(w)^{p-1}}  dA(w) \nonumber \\
&+ M_f^{(2-p)p} \left(\int_{\D \setminus r\D} |f(z)|^p \tilde{\eta}(z)dA(z) \right)^{p-1} \int_{\D}  \frac{|g_n(z)-g_n(w)|^pdA_{(p-2+s)p}(w)}{|1-\overline{w}z|^{(2+s)p}{\tilde{\eta}}(w)^{p-1}}
\end{align}
}
for all $z\in \D$ and all $n$.  By Lemma \ref{0729-2} and  the fact that $\{g_n\}$ converges to 0 uniformly on every compact subset of $\D$, there exits $N>0$ such that
\begin{align} \label{38-4}
&\ind  \int_{r\D}  \frac{|g_n(z)-g_n(w)|^p(1-|w|^2)^{(p-2+s)p}}{|1-\overline{w}z|^{(2+s)p}{\tilde{\eta}}(w)^{p-1}}  dA(w)  \tilde{\eta}(z)dA(z)\nonumber \\
\thickapprox &  \ind  \int_{r\D}  |g_n(z)-g_n(w)|^p  dA(w)  \tilde{\eta}(z)dA(z) \nonumber \\
\lesssim & \ind |g_n(z)|^p \tilde{\eta}(z)dA(z) + \int_{r\D}  |g_n(w)|^p dA(w) < \epsilon
\end{align}
for all $n>N$.
From (\ref{38-2}) and Case (a) in  the proof of  ``$(i)\Rightarrow (iii)$" of  Theorem \ref{0729-3}, we get
{\small
\begin{align} \label{38-5}
& \left(\int_{\D \setminus r\D} |f(z)|^p \tilde{\eta}(z)dA(z) \right)^{p-1} \ind  \int_{\D}  \frac{|g_n(z)-g_n(w)|^pdA_{(p-2+s)p}(w)}{|1-\overline{w}z|^{(2+s)p}{\tilde{\eta}}(w)^{p-1}} \tilde{\eta}(z)dA(z)\nonumber \\
\lesssim & \epsilon \|g_n\|^p_{B_p(\tilde{\eta}_{[2-p]})} \lesssim  \epsilon
\end{align}
}
for all $n$.
From (\ref{1021-2}), (\ref{38-3}), (\ref{38-4}) and (\ref{38-5}), we obtain
$$
\|f\overline{g_n}-\overline{h_{s,f}g_n}\|_{L^p(\D, \tilde{\eta} dA)} \lesssim  \epsilon.
$$
for all $n>N$. In other words,
 $$
 \lim_{n \to \infty}\|f\overline{g_n}-\overline{h_{s,f}g_n}\|_{L^p(\D, \tilde{\eta} dA)}=0.
 $$
Combining this with (\ref{38-1}), we see that  $\|f\overline{g_n} \|^p_{L^p(\D, \tilde{\eta} dA)} \to 0$ as $n\to \infty$. Thus  $d\nu(z)$ is  a vanishing  $p$-Carleson measure for $B_p(\tilde{\eta}_{[2-p]})$.
Case (b) and Case (c) can be proved in the similar way with minor modifications. We finish the proof.
\end{proof}
By Theorem \ref{0729-3} and the explanations    before Proposition   \ref{0728-1}, the  corresponding  results about the compactness of operators  in   Proposition   \ref{0728-1}, Theorem \ref{4main} and Theorem \ref{5main} also hold.


\begin{thebibliography}{99999}

\bibitem{AC} A. Aleman and O. Constantin, Spectra of integration operators on weighted Bergman spaces,  {\it  J. Anal. Math.}, {\bf 109} (2009), 199-231.

\bibitem{APR} A. Aleman, S. Pott, M. Reguera, Characterizations of a limiting class $B_\infty$ of B\'ekoll\'e-Bonami weights, {\it Rev. Mat. Iberoam.}, {\bf 35} (2019),  1677-1692.

\bibitem{AS} A. Aleman and A. Siskakis, Integration operators on Bergman spaces,
{\it Indiana Univ. Math. J.}, {\bf 46} (1997),  337-356.

\bibitem{AFP} J. Arazy, S. Fisher and J. Peetre, Hankel operators on weighted Bergman spaces, {\it Amer. J. Math.}, {\bf 110} (1988),  989-1053.



\bibitem{ARS} N. Arcozzi, R. Rochberg, and E. Sawyer, Carleson measures for analytic Besov spaces, {\it Rev. Mat. Iberoam.},  {\bf 18}  (2002), 443-510.

\bibitem{Ax} S. Axler, The Bergman space, the Bloch space, and commutators of multiplication operators, {\it Duke Math. J.},  {\bf 53} (1986), 315-332.

\bibitem{BGP} G. Bao, N. Gogus and S. Pouliasis, On Dirichlet spaces with a class of superharmonic weights,  {\it Canad. J. Math.}, {\bf 70} (2018),  721-741.

\bibitem{BLQW} G. Bao, Z. Lou, R. Qian and H. Wulan, On multipliers of Dirichlet type spaces, {\it Complex Anal. Oper. Theory},  {\bf 9} (2015),  1701-1732.

\bibitem{BWZ} G. Bao, H. Wulan and K. Zhu,  A Hardy-Littlewood theorem for Bergman spaces, {\it  Ann. Acad. Sci. Fenn. Math.},  {\bf 43}(2018), 807-821.

\bibitem{Bek} D. B\'ekoll\'e,  In\'egalit\'e \'a  poids pour le projecteur de Bergman dans la boule unit\'e  de $\C^n$, {\it Studia Math.}, {\bf 71} (1981/82), 305-323.

\bibitem{BB} D. B\'ekoll\'e and A. Bonami,  In\'egalit\'e \'a  poids pour le noyau de Bergman, {\it C. R. Acad. Sci. Pairs Ser. A-B}, {\bf 286} (1978), 775-778.


\bibitem{Be} A. Beurling, Ensembles exceptionnels, {\it Acta Math.}, {\bf 72} (1940), 1-13.

\bibitem{BP} D. Blasi and J. Pau,  A characterization of Besov-type spaces and applications to Hankel-type operators, {\sl  Michigan Math. J.}, {\bf 56} (2008),  401-417.

\bibitem{CM} C. Cowen and B.  MacCluer, Composition operators on spaces of analytic functions, {\sl Studies in Advanced Mathematics,  CRC Press, Boca Raton, FL,} 1995.

\bibitem{Do} J. Douglas, Solution of the problem of Plateau, {\it Trans. Amer. Math. Soc.}, {\bf 33} (1931), 263-321.

\bibitem{J} S. Janson, Generalizations of Lipschitz spaces and an application to Hardy spaces and bounded mean oscillation, {\it Duke Math. J.}, {\bf 47} (1980),  959-982.



\bibitem{LQ} Z. Lou and  R. Qian, Small Hankel operators on Dirichlet-type spaces and applications, {\sl Math. Inequal. Appl.}, {\bf 19} (2016),  209-220.

\bibitem{Mer} S.  Mergelyan, On completeness of systems of analytic functions, {\sl  Uspehi Matem. Nauk (N.S.)},  {\bf 8} (1953),  3-63.

\bibitem{OF} J. Ortega and J. F\'abrega,  Pointwise multipliers and corona type decomposition in $BMOA$, {\sl  Ann. Inst. Fourier (Grenoble)}, {\bf 46} (1996), 111-137.



\bibitem{PP} M. Pavlovi\'c and J. Pel\'aez, An equivalence for weighted integrals of an analytic function and its derivative, {\it Math. Nachr.}, {\bf 281} (2008),  1612-1623.



\bibitem{Pe} J. Pel\'aez,  Small weighted Bergman spaces, {\it Proceedings of the summer school in complex and harmonic analysis, and related topics},  2016.

\bibitem{PR2021} J. Pel\'aez and J.  R\"atty\"a, Bergman projection induced by radial weight, {\sl  Adv. Math.}, {\bf 391} (2021), paper no. 107950, 70 pp.

\bibitem{PR2014} J. Pel\'aez and J.  R\"atty\"a,  Weighted Bergman spaces induced by rapidly increasing weights, {\sl  Mem. Amer. Math. Soc.}, {\bf 227} (2014),  vi+124 pp.

\bibitem{PPR} J. Pel\'aez, A. Per\"al\"a and J.  R\"atty\"a,  Hankel operators induced by radial B\'ekoll\'e-Bonami weights on Bergman spaces,
{\it Math. Z.}, {\bf 296} (2020), 211-238.

\bibitem{PR} J. Pel\'aez and J. R\"atty\"a, Embedding theorems for Bergman spaces via harmonic analysis, {\sl Math. Ann.}, {\bf 362} (2015), 205-239.

\bibitem{PRS} J. Pel\'aez,  J. R\"atty\"a and K.  Sierra, Berezin transform and Toeplitz operators on Bergman spaces induced by regular weights, {\sl  J. Geom. Anal.}, {\bf 28} (2018),  656-687.

\bibitem{Po} S. Power, Hankel operators on Hilbert space, {\sl  Research Notes in Mathematics, 64. Pitman (Advanced Publishing Program), Boston, Mass.-London,} 1982.


\bibitem{Re} A.  Reijonen,  Besov spaces induced by doubling weights, {\sl  Constr. Approx.}, {\bf 53} (2021),  503-528.

\bibitem{RW} R. Rochberg and Z. Wu, A new characterization of Dirichlet type spaces and applications, {\sl Illinois J. Math.}, {\bf 37} (1993), 101-122.

\bibitem{WZ} H. Wulan and K. Zhu,  M\"obius Invariant $\mathcal{Q}_K$ Spaces, {\sl Springer, Cham,}  2017.

\bibitem{YZ} Ch. Yuan and H. Zeng, A double integral characterization of a Bergman type space and its M\"obius invariant subspace, {\sl Bull. Korean Math. Soc.}, {\bf 56} (2019),  1643-1653.


\bibitem {Zhao} R. Zhao, Distances from Bloch functions to some M\"obius invariant spaces,  {\it Ann. Acad. Sci. Fenn. Math.},  {\bf 33} (2008), 303-313.

\bibitem{Zhu} K. Zhu,  Operator Theory in Function Spaces, {\sl American Mathematical Society, Providence, RI,} 2007.

\bibitem{Zhu1} K. Zhu, Analytic Besov spaces, {\sl J. Math. Anal. Appl.}, {\bf 157} (1991),  318-336.

\end{thebibliography}
\end{document}